\documentclass{qt2018author}
\usepackage[utf8]{inputenc}
\usepackage[english]{babel}
\usepackage{graphicx}
\usepackage{amssymb}
\usepackage{amsmath}
\usepackage[mathscr]{euscript}
\usepackage{accents}
\usepackage{color}

\usepackage{tikz}
\usetikzlibrary{decorations.markings}

\usepackage{microtype} 

\theoremstyle{plain}	
\newtheorem{theorem}{Theorem}[section]
\newtheorem{proposition}[theorem]{Proposition}

\theoremstyle{definition} 
\newtheorem{definition}[theorem]{Definition}
\newtheorem{example}[theorem]{Example}

\newcommand{\Lbrack}{[\![}
\newcommand{\Rbrack}{]\!]}

\title{Stratified spaces, Directed Algebraic Topology, and State-Sum TQFTs}
\author{I.\ J.\ Lee and D.\ N.\ Yetter}

\address[leei@rowan.edu]{I.\ J.\ Lee \\ Mathematics Department \\ Rowan University \\ Glassboro, NJ 08028 \\U.S.A}
\address[dyetter@math.ksu.edu]{D.\ N.\ Yetter \\ Department of Mathematics \\ Cardwell Hall \\ Kansas State University \\ Manhattan, KS 66506 \\ U.S.A}

\begin{document}
\tikzset{midarrow/.style={
		decoration={markings,
			mark= at position 0.5 with {\arrow{#1}},
		},
		postaction={decorate}
	}
}

\begin{abstract}

We apply the theory of directed topology developed by Grandis \cite{Gr1, Gr2} to the study of stratified spaces by describing several ways in which a stratification or a stratification with orientations on the strata can be used to produce a related directed space structure.  This description provides a setting for the constructions of state-sum TQFTs with defects of \cite{DPY, LY}, which we extend to a similar construction of a Dijkgraaf-Witten type TQFT in the case where the defects (lower dimensional strata) are not sources or targets, but sources on one side and targets on the other, according to an orientation convention.

\end{abstract}



\maketitle

\section{Introduction}

The subject of directed algebraic topology first arose under the motivation of applications to concurrency theory in computer science \cite{FGR1, FGR2} and has been extensively developped by Grandis \cite{Gr1, Gr2} in a manner parallel to classical homotopy theory.  It has, heretofore, found little direct application to matters in what might be regarded as ``core topology''.  It is the purpose of this note to remedy this by showing that directed algebraic topology admits applications to the study of stratified spaces, including manifolds equipped with a system of submanifolds as in \cite{CY, DPY, LY}.  Throughout, ${\bf I}$ denotes the standard unit interval, and all compositions are written in diagrammatic order unless parentheses indicate functional application, thus in any category considered herein, if $t(f) = s(g)$, $fg$ and $g(f)$ both denote the composition $f$ followed by $g$.

\section{Directed Topological Spaces}

We recall from Grandis \cite{Gr1}

\begin{definition}  A {\em directed topological space}  or {\em $d$-space} $X = (X,dX)$ is a topological space $X$, equipped with a set $dX$ of continuous maps $\phi:{\bf I}\rightarrow X$ satisfying

\begin{enumerate}
\item Every constant map $x:{\bf I}\rightarrow X$ for $x \in X$ is in $dX$.
\item $dX$ is closed under pre-composition with continuous weakly monotone functions from ${\bf I}$ to ${\bf I}$.
\item $dX$ is closed under concatenation of paths.
\end{enumerate}

\noindent Elements of $dX$ are called {\em directed paths} in $X$.  Note that the second condition implies that $dX$ is both closed under (weakly monotone) reparametrizations of paths and arbitary factorizations of paths with respect to concatenation, and indeed is equivalent to being closed under these two operations, since both reparameterizations of and entire path and all parametrizations of factors under concatenation can be given by paths resulting from pre-composition with continuous monotone functions.

A {\em map of directed spaces} $f:(X,dX)\rightarrow (Y,dY)$ is a continuous function $f:X\rightarrow Y$ such that $p \in dX$ implies $f(p) \in dY$.
\end{definition}

It is easy to see that the category of directed spaces and maps of directed spaces has products given by

\[ \prod_{j\in J} (X_j, dX_j) = (\prod_{j\in J} X_j, \{(p_j)_{j\in J} | \forall j p_j \in dX_j\}) \].

Grandis \cite{Gr1} observes that topological spaces equipped with certain other auxiliary structures naturally give rise to $d$-spaces.  In particular if $X$ is equipped with a preorder (a reflexive, transitive relation) $\preceq$, letting 

\[ dX := \{\phi:{\bf I}\rightarrow X \;\; | \;\; \mbox{\rm $\phi$ is continous and } a \leq b \; \mbox{\rm implies } \phi(a) \preceq \phi(b) \}, \]

\noindent that is the set of continuous paths monotone with respect to the preorder, gives a $d$-space structure.

It will be useful to recall that a preorder $\preceq$ can be considered as equivalent to an equivalence relation, together with a partial ordering of the equivalence classes, more precisely, it is an elementary exercise to show

\begin{proposition}
If $\preceq$ is a preorder on a set $A$, the relation $\equiv$ defined by $a \equiv b$ if and only if $a \preceq b$ and $b \preceq a$ is an equivalence relation, and $\preceq$ induces a partial ordering on the set of equivalence classes by $[a] \leq [b]$ if and only if $a \preceq b$ for any representatives $a$ and $b$ of the respective equivalence class, which condition is independent of the choice of representatives.

Conversely, given an equivalence relation $\equiv$ on a set $A$ and a partial ordering $\leq$ of the equivalence clases, there is a preorder on $A$ given by $a \preceq b$ if and only if $[a] \leq [b]$.
\end{proposition}

In a similar way a weaker sort of relation than a preorder, but tied to the topology on $X$, will also induce a $d$-space structure:

\begin{definition} A {\em locally preordered space} is at topological space equipped with a reflexive relation $\preceq$ (termed a {\em precedence relation}) with the property that for every $x \in X$ there exists a neighborhood $U_x$ to which the restriction of $\preceq$ is a pre-order.
\end{definition}

Again Grandis \cite{Gr1} observes that a locally preordered space becomes a $d$-space when equipped with

\begin{eqnarray*} dX & := & \{\phi:{\bf I}\rightarrow X \;\; | \;\; \mbox{\rm $\phi$ is continous and } \\ & & \exists \epsilon > 0 \; |a-b|<\epsilon \; \mbox{\rm and } a \leq b \; \mbox{\rm implies } \phi(a) \preceq \phi(b) \}, \end{eqnarray*}

\noindent that is the set of continuous paths which are locally monotone, gives a $d$-space structure.

An aside to the reader:  the spirit of Grandis definition of the $d$-space structure from a local preorder is that every $\phi(x)$ for $x\in{\bf I}$ admits a neighborhood on which the precedence relation is a preorder and $\phi$ is monotone with respect to the precedence relation.  Because {\bf I} is compact, this can be restated with the choice of an $\epsilon$ for which $\phi$ is monotone on all open intervals of length $\epsilon$ or less.

Grandis \cite{Gr1} states as a passing observation a result which we will have cause to use, and therefore state as a proposition slightly extended by a more constructive description:

\begin{proposition} The set of $d$-space structures on a topological space $X$ forms a complete lattice with meet given by intersection of the sets of directed paths.  The join of a set of $d$-space structures $d_\iota X$
for $\iota \in {\mathcal I}$ for some indexing set $\mathcal I$ is given by

\begin{center}
\parbox{8cm}{

$(\bigvee_{\iota \in {\mathcal I}} d_\iota)X = $ 

\hspace{1cm}$\{ p:{\bf I}\rightarrow X\; |\; \exists 0 = a_0 \leq a_1 \leq \ldots \leq a_n = 1 $

\hspace{1.3cm}$\exists \iota_1, \ldots \iota_n \in {\mathcal I} \;\; p|_{[a_{j-1},a_j]} \in d_{\iota_j}X \} $ }
\end{center}

\noindent where restrictions of paths to subintervals are understood as paths with source {\bf I} via some (and therefore any) monotone parametrization of the subinterval.
\end{proposition}

\begin{proof}
As Grandis observes, it is trivial to see that the set of $d$-space structures on $X$ is partially ordered by containement, is closed under arbitrary intersections, and has a maximal element, the set of all paths in $X$.
This leaves the construction of joins to be the usual ``theological'' construction -- intersect all upper bounds -- that gives no insight into exactly what the join looks like.

To see that our constructively describe $\bigvee d_\iota$ is the join, observe that any directed space structure containing all the $d_\iota X$'s, being closed under concatenation, must contain the paths specified in our description.  Our set of paths also, trivially contains the constant paths and is closed under concatentation.  It remains only to see that it is closed under precomposition by arbitrary continuous (weakly) monotone maps from {\bf I}.  But this is easy to establish since the pullbacks of the intervals in the decomposition will give a new decomposition of {\bf I}, and precomposing each of the maps of subintervals with a parametrization by {\bf I} will give a directed path in the same $d_\iota X$ as the subinterval pulled back.
\end{proof}

We will not have need for the full breadth of Grandis's development of directed algebraic topology \cite{Gr2} in the present work, only the notion of fundamental category, which we now briefly describe in slightly different terms than those given in \cite{Gr1, Gr2}. 

First it is trivial to establish: 

\begin{proposition}  Let $\sim$ denote the equivalence relation on paths induced by weakly monotone reparametrization.  Then for any directed space $(X,dX)$, 

\[ {\mathcal C}_{(X,dX)} := (X, dX/\sim, -(0), -(1), *) \]

\noindent is a category, where $-(j)$ denotes evaluation at $j$, and $*$ denotes concatentation of reparametrization-classes of paths.  
\end{proposition}

Let $\uparrow\! {\bf I}$ denote ${\bf I}$ with the directed space structure given by all (weakly) monotone increasing maps.

\begin{proposition} If $(X,dX)$ is a directed space, then there is a bijection between $dX$ and the
set of maps of directed spaces from $\uparrow\! {\bf I}$ to $(X,dX)$ given by

\[ p:\uparrow\! {\bf I}\rightarrow (X,dX)\;\; \longleftrightarrow\;\; Up: {\bf I}\rightarrow X \in dX \]

\noindent where $U$ denotes the underlying functor from $d$-spaces to topological spaces.
\end{proposition}

\begin{proof}
Plainly if $p$ is a map of directed spaces, in particular it takes the identity map of $I$ to a directed path in $X$, namely $Up$.  Conversely given a directed path $\pi \in dX$, we need to show that $\pi$ maps every directed path in $\uparrow\! {\bf I}$ to a directed path in $X$.  But directed paths in $\uparrow\! {\bf I}$ are exactly weakly monotone maps from {\bf I} to {\bf I} so this follows from the closure conditions on $dX$.
\end{proof}

We hereinafter identify the elements of $dX$ with $d$-space maps from 
$\uparrow\! {\bf I}$ to $(X,dX)$.

\begin{definition}
Two parallel arrows $[p_0]$ and $[p_1]$ in ${\mathcal C}_{(X,dX)}$ are {\em immediately $d$-homtopic} if there exists a map of directed spaces $h: \uparrow\! {\bf I}^2\rightarrow (X,dX)$ such that $h(-,i)$ represents $[p_i]$ for $i=0,1$. 

Two parallel arrows $[p_0]$ and $[p_1]$ are {\em $d$-homtopic} if they are equivalent under the equivalence relation generated by immediate $d$-homotopy.  We write $\Lbrack p \Rbrack$ for the $d$-homotopy class represented by a directed path $p$ and $\simeq$ for the relation of $d$-homotopy.
\end{definition}

When it is observed that immediate $d$-homotopies can be concatenated in first coordinate and that there is always a trivial immediate $d$-homotopy from any reparametrization class of paths to itself, it is easy to see that $d$-homotopy is not merely an equivalence relation, but a congruence of categories.

\begin{definition}
The {\em fundamental category} of a directed space $\uparrow\! \Pi(X,dX)$ is the quotient category of ${\mathcal C}_{(X,dX)}$ by the congruence $\simeq$.
\end{definition}

\section{Filtered and Stratified Spaces}

As our purpose is to connect Grandis's work \cite{Gr1, Gr2} to more traditional notions in topology, we now introduce those, beginning with a more general notion of filtration than is usually given.

\begin{definition}
If $({\mathcal O}, \leq)$ is a totally ordered set with maximum element $\infty$, an $\mathcal O$-filtration of a topological space $X$ is a set of subspaces $X_k$ for $k \in {\mathcal O}$ such that 
\begin{enumerate} 
\item $X_\infty = X$ and
\item $k \leq \ell$ implies $X_k \subseteq X_\ell$.
\end{enumerate}
\end{definition}

It is most common for $\mathcal O$ to be $\{ 0,1,\ldots n\}$ for some natural number $n$ with the usual ordering (so that $n$ is $\infty$ in the general description).  

An $\mathcal O$-filtered space becomes preordered space, and thus a $d$-space in an obvious way:
let $x \leq y$ whenever  $x\in X_k$ and $y\in X_\ell$ for $k \leq \ell$.

We refer to this as the {\em source preorder} associated to a filtration.  The opposite preorder will be called the {\em target preorder}.

Among filtered spaces, the most important to geometric topology are stratified spaces.  There are many definitions of stratified space in the literature.  Some require neighborhoods of points in each part of the filtration admit neighborhoods of a particular form -- a product of an open in the stratum and a cone on a manifold.  In differential topology, the ``Whitney conditions'' on the behavior of tangent spaces at and near strata are usually required.  For our purposes one of the simplest will suffice:

\begin{definition}
A {\em stratified space} is a $(\{0,1,\ldots n\}, \leq)$-filtered space with $X_k\setminus X_{k-1}$ a $k$-manifold, and $\overline{X_k} \setminus X_k \subset X_{k-1}$.  A {\em stratum} of a stratified space is a connected component of some $X_k\setminus X_{k-1}$.  We say two strata $S_i$ and $S_j$ are {\em adjacent} if one is contained in the closure of the other.
\end{definition}

In \cite{DPY} and \cite{LY} defect TQFTs generalizing Dijkgraaf-Witten theory to certain, very simple, stratified spaces ($\emptyset = X_0 \subset X_1 \subset X_2 = X$, for $X$ a closed oriented surface and $X_1$ a curve contained in it, and $\emptyset = X_0 \subset X_1 \subset X_2 \subset X_3 = X$, for $X$ a 3-manifold, $X_1$ a link contained in it, and $X_2$ a Seiftert surface to the link, respectively) were constructed.  In both cases the ``untwisted'' theory could be interpreted as counting functors from a skeletal category equivalent to the fundamental category of the $d$-space given by the source preorder, so that paths may leave lower-dimensional strata, but not enter them.

In the stratified setting, as opposed to the more general filtered setting, there are other ways in which directed space structures can be induced.  For instance in the circumstance of \cite{LY} if the ``knot'' is actually a link, some of the components could be sources and others targets.  In particular, for two component links (or links with their components partitioned into a first sublink and a second) a theory corresponding to that of [5] can be developped in which the first component of the link and the Seifert surfaces are source defects, while the second component of the link is a target defect.  This, again is described by a preordered space, the preorder being chaotic on each stratum, and ordering the strata with the first component of the link less than the Seifert surface, less than the bulk, less than the second component of the link, the colorings for a Dijkgraaf-Witten theory in this setting would lie in a {\textsf 4}-parcel in the terminology of \cite{LY}.


More flexible ways of constructing a directed space from stratified space are available by tying a local preorder to the stratification.

\begin{definition}
If $X$ is a stratified space with strata $S_i$ for $i \in {\mathcal I}$, a {\em narrow cover} is an open covering
by open sets $U_i$ for $i \in {\mathcal I}$ such that

\begin{enumerate}
\item $S_i \subset U_i$ for all $i\in {\mathcal I}$,
\item $U_i$ is connected for all $i \in {\mathcal I}$,
\item If $U_i \cap S_j \neq \emptyset$ then either $i = j$ or the dimension of $S_j$ is greater than the dimension of $S_i$ and $S_i$ and $S_j$ are adjacent.  
\end{enumerate}
\end{definition}

For a stratum $S_i$, we will denote the set of strata adjacent to $S_i$ with dimension greater than or equal to the dimension of $S_i$ by $N_i$, so the last condition may be rephrased as for all $S_j$,  $U_i \cap S_j = \emptyset$ or $S_j \in N_i$.

\begin{proposition} If $X$ is a stratified space with strata $S_i$ for $i \in {\mathcal I}$, then the set of narrow covers ${\bf Nar}(X)$ is partially ordered by $\{U_i \} \leq \{V_i\}$ whenever $\forall i\in{\mathcal I}$ $U_i \subset V_i$
and forms a meet semilattice under  $\{U_i \} \wedge \{V_i\} = \{W_i\}$ for $W_i$ the connected component of $U_i \cap V_i$ which contains $S_i$.
\end{proposition}

\begin{proof}
That the relation is a partial ordering is clear.


By construction $\{W_i\}$ satisfies all three conditions to be a narrow cover and is a lower bound on $\{\{U_i \}, \{V_i\}\}$.  But enlarging any of the opens in $\{W_i\}$ either produces an open cover which fails condition (2) or a narrow cover which is not less than one of $\{U_i\}$ or $\{V_i\}$, so that $\{W_i\}$ is the meet.
\end{proof}

\begin{definition}
If $X$ is a stratified space with strata $S_i$ for $i \in {\mathcal I}$, a {\em stratified local preorder subordinate to a narrow cover $U_i$ for $i \in {\mathcal I}$} is a local preorder, $\preceq$, the restriction of which to each $U_i$ is a preorder, with the equivalence classes of the associated equivalence relation being unions of connected components of the sets $U_i \cap S_j$ for $S_j$ ranging over the strata.
\end{definition}


In fact, the notion of stratified local preorder is independent of the choice of narrow cover in a sense made precise by

\begin{proposition}
If $\preceq$ is a stratified local preorder subordinate to the narrow cover $\{U_i | i \in {\mathcal I}\}$, and $\{V_i | i\in {\mathcal I}\}$ is a narrow cover with $V_i \subseteq U_i$ for all $i\in {\mathcal I}$, then $\preceq$ is a stratified local cover subordinate to $\{V_i | i\in {\mathcal I}\}$.
\end{proposition}

\begin{proof}

Now for each stratum $S_i$, since $V_i \subseteq U_i$ it is clear that the restriction of $\preceq$ to $V_i$ is a preorder, since the restriction of a preorder to a subset is again a preorder.  

It remains to show that the equivalence classes induced by the restriction are unions of connected components of sets of the form $V_i \cap S_j$ for $S_j$ ranging over the strata.  Now suppose $x \in V_i$.  Let $[x]_U$ denote the equivalence class of $x$ under the symmetrization of the restriction of $\preceq$ to $U_i$, and similarly $[x]_V$ the equivalence class of $x$ under the symmetrization of the restriction of $\preceq$ to $V_i$.  By hypothesis $[x]_U = \cup_{j,k} C_{j,k}$, where the $C_{j,k} \subset U_i \cap S_j$ are among the connected components of $U_i \cap S_j$.  Now $V_i \subseteq U_i$, and it is easy to see that $[x]_V = V_i \cap [x]_U$ since the restriction of the symmetrization of a relation is the symmetrization of the restriction.  Thus
$[x]_V = \cup_{j,k} (V_i \cap C_{j,k})$  But, as $V_i \subseteq U_i$, it follows that $V_i \cap S_j \subseteq U_i \cap S_j$, so $V_i \cap C_{j,k}$, while it need not be connected, is necessarily a union of connected components of $V_i \cap S_j$, and we are done.

\end{proof}

We will thus refer to a local preorder as a stratified local preorder if there exists a narrow cover $\{U_i\}$ for which it is a stratified local preorder subordinate to $\{U_i\}$.

On the other hand, in the circumstance of the previous proposition, there may be stratified local preorders subordinate to $\{V_i | i\in {\mathcal I}\}$ which are not subordinate to $\{U_i | i \in {\mathcal I}\}$:  consider the stratified space $\{1\} \subset S^1$.   The only stratified local preorders subordinate to the maximal narrow cover, $\{ S^1, S^1\setminus \{1\}\}$, are the actual preorders with $\{1\}$ and $S^1\setminus \{1\}$ as equivalence classes.  However, for the narrow cover $\{ \{e^{\theta i} | \theta \in (-1,1)\}, S^1\setminus \{1\}\}$, the relation given by --
$e^{\theta i }\preceq e^{\phi i}$ for $\theta, \phi \in [-\pi, \pi)$ exactly when $\theta \neq 0 \neq \phi$ or  $-1 < \theta \leq \phi < 1$ -- is also a stratified local preorder subordinate to $\{ \{e^{\theta i} | \theta \in (-1,1)\}, S^1\setminus \{1\}\}$, as is the relation given by the same conditions with $\phi$ and $\theta$ interchanged.

This last example readily generalizes to the case of a manifold with an orientable submanifold of codimension 1:  Choose a bicollar neighborhood of the submanifold.  The bicollar neighborhood and the complement of the submanifold form a narrow cover.  Now preorder the bicollar neighborhood with a preorder which is chaotic on each collar and on the submanifold, but for which one collar is less than the submanifold which is less than the other collar, and extend it to a precedence relation on the whole manifold by taking its union with the chaotic equivalence relation on the complement of the submanifold.  In cases where the submanifold does not separate (or has a component which does not separate), as with the point on the circle, the result will be a stratified local preorder, but not a preorder.

In cases where the ambient manifold and submanifold are both oriented, the choice of ordering can be made on the basis of orientations:  let the submanifold be greater than the collar for which the orientation of the submanifold is the boundary orientation, and less than the collar for which its orientation is opposite the boundary orientation. 

We will next consider how this last construction can be generalized to stratified spaces with oriented strata, and will see that while this will always give rise to a directed space, only in some instances will it give rise to a locally preordered space.  In order for orientations on strata to agree or disgree with boundary orientations induced by the orientations on higher dimensional strata, it is necessary for the stratification to be of a restricted type:  the closures of strata must be manifolds with boundary, and more than that, all strata must admit neighborhoods which are unions of collars.   This precludes many interesting examples of statified spaces, as, for instance complex varieties stratified by iterated singular sets, to which the theory developed thus far applies.  It does, however, include the sort of examples from low-dimensional geometric topology in \cite{CY, DPY, LY}.  To be precise we define a topological analogue of the stark neighborhoods in the PL theory developped in \cite{CY}.

\begin{definition}  A {\em collared stratification} of a space $X$ is a stratification  $X_0 \subset \ldots \subset X_n = X$ in which the closure of each stratum $S \subseteq X_k$ is a (topological) $k$-manifold with boundary, and, moreover $X$ admits a narrow cover $\{U_i\}$ in which for each open $U_i $  and each stratum $S_j \subset X_d$, $U_i \cap S_j$ admits a neighborhood in $U_i \cap X_{d+1}$ homeomorphic to $(U_i \cap S_j) \times CF$, for $CF$ the cone on a finite set of points, $F$, and moreover when $S_j = S_i$, this neighborhood may be taken to be all of $U_i \cap X_{d+1}$ 

We call such a narrow cover a {\em multicollar (narrow) cover}.
\end{definition}

Now given a collared stratification and a choice of orientation for each stratum, we can define two relations on the strata of each $U_i$ in a multicollar cover $\{U_i\}$ induced by the orientations:

\begin{definition} If $S$ and $T$ are two oriented strata of the restriction of the filtration to $U_i$, we say $T$ boundary covers $S$ (resp. $S$ anti-boundary covers $T$) denoted $S \stackrel{\prec}{\cdot}_+ T$ (resp. $T\stackrel{\prec}{\cdot}_- S$) if 

\begin{itemize}
\item $\dim(T) = \dim(S) -1$
\item $S$ and $T$ are adjacent
\item the boundary orientation induced on $T$ by the orientation of $S$ agrees with the orientation of $T$ (resp. the boundary orientation induced on $T$ by the orientation of $S$ is opposite to the orientation of $T$).
\end{itemize}
 
These relations then induce relations on the points of $U_i$ for which we use the same notation:  $x \stackrel{\prec}{\cdot}_+ y$ (resp. $x \stackrel{\prec}{\cdot}_- y$) when the strata they lie in satisfy the relation denoted by the same symbol.

The transitive reflexive closures of the relations $\stackrel{\prec}{\cdot}_+$ and $\stackrel{\prec}{\cdot}_-$ each induce partial orderings of the strata of $U_i$, which we denote $\preceq_{U_i, +}$ and $\preceq_{U_i,-}$, respectively and corresponding preorderings on the elements of $U_i$, which are chaotic on the strata of $U_i$.

\end{definition}

In general these partial orderings of the local strata may not only have incomparable elements, but may even have disconnected incidence diagrams.  Observe that chains of local strata with respect to $\preceq_{U_i,+}$ decrease in dimension, while chains of local strata with respect to $\preceq{U_i,-}$ increase in dimension.

Using these notions, we can construct several different $d$-space structures on a space with a collared stratification and oriented strata.  Observe that taking the union of the chaotic preorders on the strata with the relations $\preceq_{U_i,+}$, (resp. $\preceq_{U_i,-}$) on each $U_i$, extended to the whole space, gives a local preorder, and thus a directed space structure $d_{+or}(X)$ (resp. $d_{-or}(X)$).  We refer to this as the {\em descending (resp. ascending) oriented $d$-space structure}.

From these and the stratification, we can define two more $d$-space structures.

\begin{definition}
Let $X$ a space equipped with a collared stratification and orientations of its strata.  Let $\{U_i\}$ be multicollar narrow cover $\{U_i\}$.  

The {\em full oriented $d$-space structure} on $X$, $d_{for}(X)$ is the join of the descending and ascending oriented $d$-space structures.

The {\em restricted  oriented $d$-space structure} on $X$, $d_{ror}(X)$  is given by the set of all paths for which there exists $0 = a_0 < a_1 <\ldots < a_n = 1$ for which each $[a_k, a_{k+1}]$ is mapped into some $U_i \cap (S_i \cup S_j)$ for $\dim(S_j) = \dim(S_i) + 1$
and the restriction of the path $[a_k, a_{k+1}]$ is either monotone increasing with respect to $\preceq_{U_i,+}$ or monotone increasing with respect to $\preceq_{U_i,-}$.
\end{definition}

Neither structure is typically given by a local preorder, once there are strata of codimension greater than 1.    For the restricted oriented $d$-space structure, this is easy to see, since monotonicity with respect to a local preorder cannot enforce the dimension restriction.

For the full oriented $d$-space structure, we will illustrate the point with a few local models.    We suppress the mention of the neighborhood $U_i$ in the subscripts throughout the examples.  In the first, there is, in fact, a (local) preorder, which induces the $d$-space structure:

\begin{example}
Consider the stratification of (a neighborhood of the origin in) ${\mathbb R}^2$ with the strata depicted in the figure below, assuming the counterclockwise orientation on the plane, the origin positively oriented and the axes oriented as indicated by the arrows.

$$\begin{array}{c}
  \includegraphics[width=3.5in]{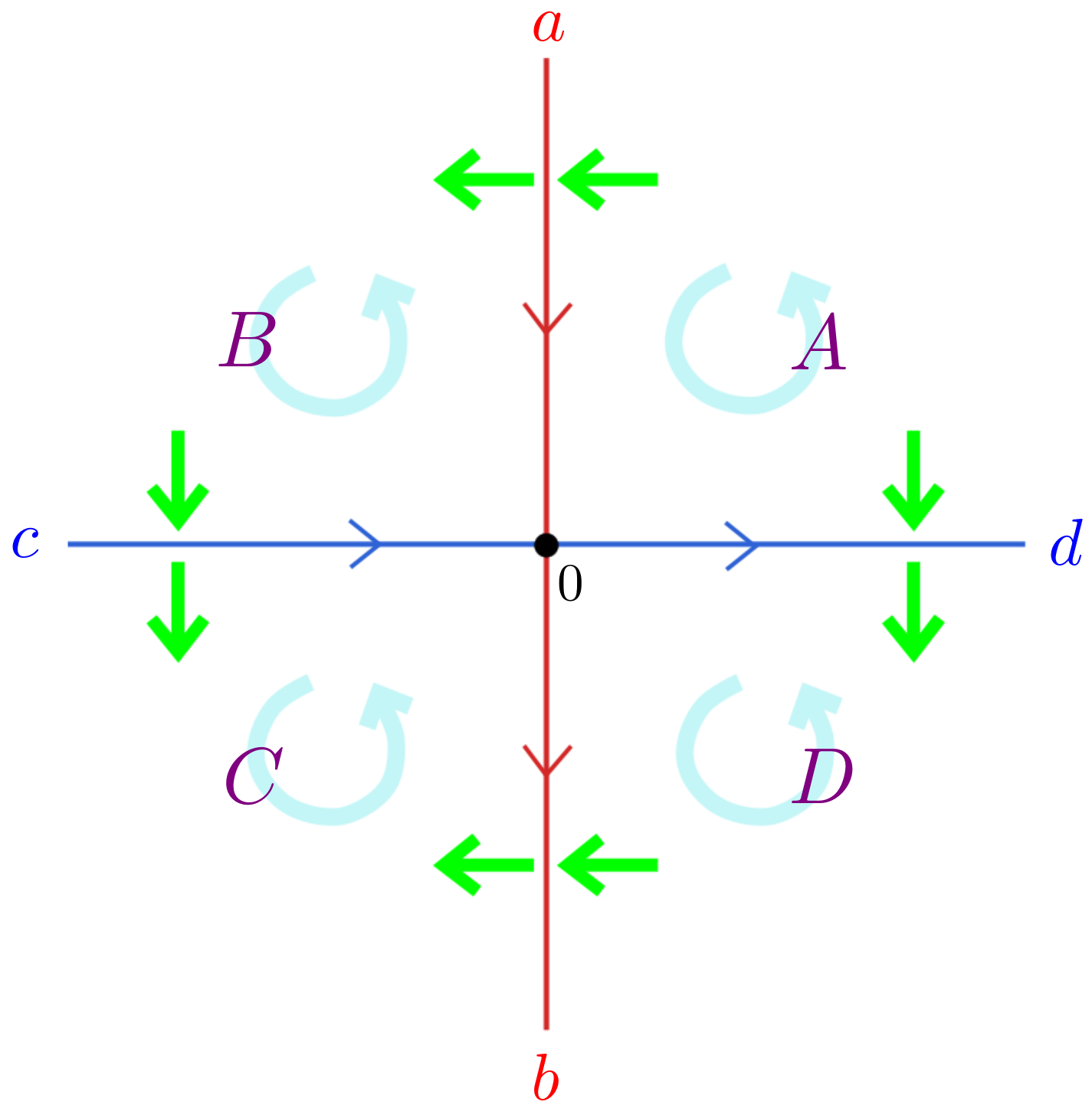}
\end{array}$$

The orientations induce the following local (strict) orderings of the strata in any multicollar neighborhood of the stratum listed between the other two. 
\begin{center}
\begin{tabular}{ll}
 $\displaystyle a \preceq_+ 0 \preceq_- b \,\, ,$ & $\displaystyle c \preceq_+ 0 \preceq_- d $\\
 \\
 $\displaystyle A \preceq_+ a \preceq_- B \,\, , $ & $\displaystyle A \preceq_+ d \preceq_- D $\\
 \\
 $\displaystyle B \preceq_+ c \preceq_- C \,\, , $ & $\displaystyle D \preceq_+ b \preceq_- C $\\ 
\end{tabular}
\end{center}

It is easy to see that multicolar neighborhoods of the axes are preordered by the ordering of the strata above with preorders chaotic on the strata.  In this case a multicolar neighborhood of the origin is also preordered by the total ordering of strata

\[ A < a < B < c < 0 < d < D < b < C .\]

It is easy to see that $d$-space structure given by this preorder is $d_{for}$:  decomposing a path monotone with respect to this preorder into subintervals with endpoints where the path changes stratum gives the required decomposition (the monotonicity requirements on the decomposition following from the global monotonicity).  The converse follows since here all instances of $\preceq_{U_i,+}$ and $\preceq_{U_i,-}$ (as relations on strata) assembled into the total ordering of the strata.

\end{example}

However, it is not always the case that $d_{for}$ arises from a local preordering.

\begin{example}
Consider the stratification of (a neighborhood of the origin in) ${\mathbb R}^2$ with the strata depicted in the figure below, assuming the counterclockwise orientation on the plane, the origin positively oriented and the axes oriented as indicated by the arrows.

$$\begin{array}{c}
  \includegraphics[width=3in]{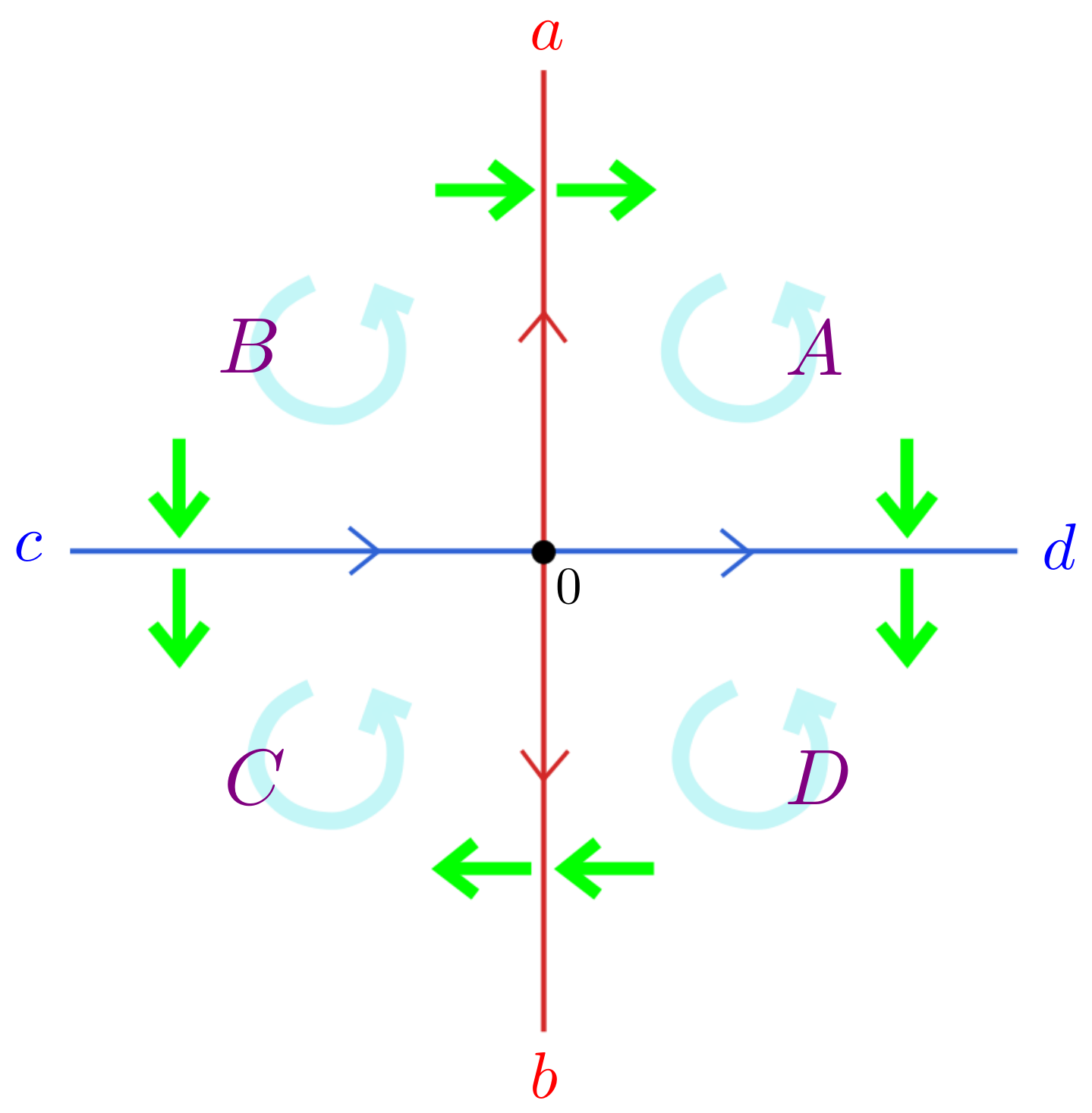}
\end{array}$$

Again the orientations induce the following local (strict) orderings of the strata in any multicollar neighborhood of the stratum listed between the other two: 
\begin{center}
\begin{tabular}{ll}
 $\displaystyle 0 \preceq_- a \,\, , \,\, 0 \preceq_- b \,\, , $ & $\displaystyle c \preceq_+ 0 \preceq_- d $\\
 \\
 $\displaystyle A \preceq_+ a \preceq_- B \,\, , $ & $\displaystyle A \preceq_+ d \preceq_- D $\\
 \\
 $\displaystyle B \preceq_+ c \preceq_- C \,\, , $ & $\displaystyle D \preceq_+ b \preceq_- C $\\ 
\end{tabular}
\end{center}

However, in this case, there is no local preorder which induces $d_{for}$ structure, as the reflexive-transitive closure of the orderings of the strata in any neighborhood of the origin will be the chaotic equivalence relation.
\end{example}

On the basis of the two examples, one might be led to believe that in cases where the filtration is given by oriented immersed submanifolds, the directed space structure induced by the orientations would be given by a local preorder.  This, however, is not the case, as the following example in 3 dimensions shows

\begin{example}
Consider the filtration of ${\mathbb R}^3$,  $X_0 \subset \ldots \subset X_3$, with $X_3$ being all of ${\mathbb R}^3$, $X_2$ the union of the coordinate planes, $X_1$ the union of the coordinate axes, and $X_0$ the origin. If ${\mathbb R}^3$ is given the right-hand rule orientation, the coordinate planes are oriented counter-clockwise when viewed from the first octant and axes oriented from $-\infty$ to $\infty$ as indicated by the arrows.  

In this case,  denoting the semi-axes by the name of their coordinate, and the quadrants by the names of their two coordinates, we have 
 
\[ x \preceq_- xy \preceq_+ y \preceq_- yz \preceq_+ z \preceq_- xz \preceq_+ x  \]

\noindent from which it immediately follows that no preorder in a neighborhood of the origin will induce
the full oriented $d$-space structure.
\end{example}

The detailed structure of the fundamental category is easier to work out for $d_{ror}$ than for $d_{for}$ because the condition that directed paths not skip dimensions limits both the paths and the homotopies in such a way that the sequence of strata traversed is invariant under directed homotopies.  However, for the same reason, constructions involving flag-like triangulations with the requirement that the edges be directed and faces carry a directed homotopy from the two-edge path to the one-edge path cannot be applied to $d_{ror}$ if there are strata of codimension greater than one.  For $d_{for}$ constructions with flag-like triangulations can be applied, but it is unclear at this writing what the appropriate class of categories to use for colorings would be.

In the next section we describe a generalized Dijkgraaf-Witten theory for stratified spaces in the one case where $d_{for}$ and $d_{ror}$ coincide:  spaces in which the stratification has only two non-empty levels:  
$\emptyset = X_{d-2} \subset X_{d-1} \subset X_d$.   Note that for any such stratified space, $X_{d-1}$ is necessarily a $d-1$-manifold and the stratificaiton is necessarily collared.  We will refer to such a directed space as an {\em oriented $d$-manifold with defect} and require that both $X_d$ and $X_{d-1}$ be compact.

\section{Counting Invariants and Generalized Dijkgraaf-Witten Theory}

For oriented manfolds with defect, the analogues of the parcels over partially ordered sets of \cite{LY} will be given by

\begin{definition}
 A {\em $\Gamma_2$-parcel} is a category  $\mathcal C$ equipped with a conservative functor $\gamma$, surjective on arrows, to the path category ${\mathcal P}(\Gamma_2)$ of the directed graph $\Gamma_2$ with vertices $\delta$ and $\beta$ (for ``defect'' and ''bulk'', respectively) and edges $\omega= (\delta, \beta)$ and $\iota = (\beta, \delta)$ (for ``outbound'' and ``inbound'', respectively). A $\Gamma_2$-parcel is {\em fiber-finite} if the inverse image of any arrow in the path category is a finite set of arrows in $\mathcal C$, and {\em regular} if it is injective on objects.
\end{definition}

As an aside, it appears that the appropriate level of generality for base categories in the notion of parcel is a gaunt category, that is, one in which all isomorphisms are identity arrows.  

The following proposition explains why $\Gamma_2$-parcels are are well adapted to studying oriented manifolds with defect:

\begin{proposition}
If $(X,d)$ is an oriented manifold with defect, and $x_i \in S_i$ is a choice of base points, one in each stratum of $X$, then the full subcategory $\uparrow \pi_1(X,d,\{x_i\})$ of $\uparrow \Pi_1(X,d)$ with the  base points $x_i$ as objects admits a canonical $\Gamma_2$-parcel structure, with the functor to ${\mathcal P}(\Gamma_2)$ given on objects by sending vertices in $d$-dimensional strata to $\beta$ and those in $d-1$-dimensional strata to $\delta$.
\end{proposition}

\begin{proof}
First observe that any directed path from one of the chosen base points to another (possibly equal) traverses a sequence of strata.  This sequence of strata will determine the image under the structural functor for the parcel structure.   

 The key to this is the observation that a directed homotopies of paths rel boundary cannot neither push a path into the defect nor pull it off.   For paths passing through the defect, it suffices to observe that homotopies rel boundary of paths will preserve the intersection number of the path with any stratum of the defect, which by construction, for directed paths will always be non-negative, and thus positive if the path passes through the defect.  Paths in the bulk cannot be homotopic to paths passing though the defect, since the moment a homotopy moved the path into the defect, it would both enter and leave on the same side, and thus not be directed.  Likewise paths in the defect are not homotopic to paths that leave and reenter the defect, since again, these paths would leave and enter on the same side, and thus no be directed.

It thus follows that the sequence of strata traversed by a path is invariant under directed homotopy.  This, together with the fact that within a statum homotopy and directed homotopy are identical shows that any path between chosen basepoints, will, either lie entirely in a single statum, or up to directed homotopy will factor as a sequence of directed paths each of which begins at a base point, and ends at a base point of an adjacent stratum.  Mapping paths within a stratum of the bulk (resp. defect) to the identity map $\beta$ (resp. $\delta$) and mapping paths from the base point in the bulk (resp. defect) to the base point of an adjacent stratum in the defect (resp. bulk) to $\iota$ (resp. $\omega$) then determines a functor from $\uparrow \pi_1(X,d,\{x_i\})$ to ${\mathcal P}(\Gamma_2)$, which is plainly conservative and surjective on arrows.
\end{proof}

We call $\uparrow \pi_1(X,d,\{x_i\})$ equipped with this functor to ${\mathcal P}(\Gamma_2)$, {\em the fundamental parcel} of the oriented manifold with defect $(X,d)$ at base points $\{x_i\}$.

It is trivial then to define a counting invariant of oriented manifolds with defect:

Given a regular fiber-finite $\Gamma_2$-parcel $\mathcal C$, let ${\mathcal N}_{\mathcal C}(X,d)$ be the number of functors from $\uparrow \pi_1(X,d,\{x_i\})$ to $\mathcal C$ such that the triangle formed with the structural functors to ${\mathcal P}(\Gamma_2)$ commutes.  By abuse of notation, we denote the unique object of a regular $\Gamma_2$-parcel mapped to $\beta$ (resp. $\delta$) by the structure functo by $\beta$ (resp. $\delta$).

This number is independent of the choice of base points for the strata, since another choice of base points will result in an isomorphic fundamental parcel.

While this invariant is defined using only stratification and the oriented directed space structure, as with the untwisted versions of Dijkgraaf-Witten theory for source defects of \cite{DPY, LY}, it can be computed from a fine enough flag-like triangulation of $(X,d)$, and will thus admit a relative version as a topological quantum field theory with defects (cf. \cite{Y2}).  In this form, as with the theories in \cite{DPY, LY} it will allow for a version ``twisted'' by suitable analogues of the group cocycles in finite guage group Dijkgraaf-Witten theory.

To construct the counting invariant as a state-sum, we require that the triangulation of $(X,d)$ be flag-like with respect to the stratification, in the present context simply meaning that the strata are subcomplexes and that every simplex intersects $X_{d-1}$ in the empty set or in a face (of any codimension).  The usual expedient of orienting edges by an ordering of the vertices does not work in this setting, instead, order the vertices of each stratum, and orient the edges that lie entirely with a stratum from earlier vertex to later, while orienting the remaining edges (with exactly one vertex in $X_{d-1}$) so that traversed in the direction of the orientation, the edge is a directed path.  Notice that although there is no global ordering of the vertices, the orientations on the edges of a top-dimensional simplex will induce a total ordering of the vertices of the simplex, and a sequence of edges joining successive vertices in order from least to greatest, which we call `` the long path''.

We call a flag-like triangluation equipped with such orientations on its edges a {\em directed triangulation}.
It is easy to see that for any extended Pachner move in the sense of \cite{CY}, once any new vertex is inserted into the ordering of the vertices in the statum in which it lies, the new triangulation formed admits orientations of all new edges so that it is a directed triangulation.

Now observe that for any directed triangulation, any 2-dimensional face is of one of six types:  lying entirely in the bulk, lying entirely in the defect, with one edge in the defect and edges oriented from the opposite vertex in the bulk to its vertices, with one edge in the defect and edges oriented from its vertices to the opposite vertex in the bulk, with one edge in the bulk and edges oriented from its vertices to the opposite vertex in the defect, or with one edge in the bulk and edges oriented to its vertices from the opposite vertex in the defect.  In each case, two of the edges, regarded as directed paths, are composable and the third edge, as a directed path, is immediately directed homotopic to them. 

We can then make

\begin{definition}
Given a directed triangulation $\mathcal T$ of $(X,d_{ror} = d_{for})$, and a regular fiber-finite $\Gamma_2$-parcel $\mathcal C$, a {\em $\mathcal C$-coloring of $\mathcal T$} is an assignment to each edge of an arrow in $\mathcal C$ in which 
\begin{enumerate}
\item the arrow assigned to an edge in the bulk lies above the identity arrow of $\beta$
\item the arrow assigned to an edge in the defect lies above the identity arrow of $\delta$
\item the arrow assigned to an edge with initial vertex in the bulk and terminal vertex in the defect lies above
the edge $\iota$ (regarded as a generating arrow of the path category)
\item the arrow assigned to an edge with initial vertex in the defect and terminal vertex in the bulk lies above
the edge f$\omega$ (regarded as a generating arrow of the path category)
\item for every 2-dimensional face, the composition of the arrows assigned to the two composble edges is equal to the arrow assigned to the third edge.
\end{enumerate}
\end{definition}

Note that as a result of the last condition, the restriction of a coloring to the edges of the long path of a simplex determines the arrows assigned by the coloring to every edge of the simplex.

We then have

\begin{theorem}
If $\mathcal T$ is any directed triangulation of $(X,d)$, letting $v_\beta$ (resp. $v_\delta$) denote the number of vertices in the bulk (resp. defect), and $\sigma_\beta$ (resp. $\sigma_\delta$) denote the number of strata in the bulk (resp. defect), then ${\mathcal N}_{\mathcal C}(X,d)$ is equal to the number of $\mathcal C$-colorings of $\mathcal T$ divided by 

\[ (\# \gamma^{-1}(Id_\beta))^{v_\beta - \sigma_\beta}\cdot
(\# \gamma^{-1}(Id_\delta))^{v_\delta - \sigma_\delta}\]
\end{theorem}

\begin{proof}
For each stratum, choose a vertex of the triangulation lying in the stratum to be a base point, and a spanning tree for the graph formed by the edges lying in the stratum.  We claim there is a bijection between $\mathcal C$-colorings of $\mathcal T$ and a pair of a $\mathcal C$-coloring of the edges the spanning forest and a map of $\Gamma_2$-parcels from the fundamental parcel with the chosen base points to $\mathcal C$.  Plainly establishing this claim suffices to establish the result.

Clearly, a $\mathcal C$-coloring of $\mathcal T$ restricts to a $\mathcal C$-coloring of the spanning forest, but it also determines a map of $\Gamma_2$-parcels from the fundamental parcel with the chosen base points to $\mathcal C$, since any path is directed homotopic to one lying in the 1-skeleton, and, moreover, all equations between elements of the fundamental parcel named by paths in the 1-skeleton are induced by the directed homotopies across 2-cells.

Conversely, given a $\mathcal C$-coloring of the spanning forest, and a map of the fundamental parcel to 
$\mathcal C$, there is a unique $\mathcal C$-coloring of $\mathcal T$, determined as follows:  Consider an edge not in the spanning forest.  Its endpoints lie in the spanning forest, each on some stratum.  Since edges within strata are directed paths when traversed in either direction, the edge, as a directed path, determines a unique directed path from the basepoint of the stratum in which its initial vertex lies to the base point in which its terminal vertex lies by pre- and post- composing with the directed paths along the edges of the spanning forest.  This path has an image in $\mathcal C$ under the parcel map.  Color the edge with the arrow of $\mathcal C$ obtained from the image in $\mathcal C$ of this path by pre- and post- composing with the inverses of the arrows obtained by composing the
edge lables (or their inverses as appropriate) along the paths in the forest.

It is easy to see that these two constructions are inverse to each other, and thus establish the desired bijection.

\end{proof}

Having shown that the invariants counting maps from the fundamental parcel to a given regular fiber-finite $\Gamma_2$-parcel arise from a state-sum on triangulations, it follows from the general constructions of \cite{Y2} that there is a topological quantum field theory with defects for which these invariants are the vacuum to vacuum expectations.

The construction just give, is, of course, an adaptation to the presence of oriented codimension 1 defects with the (full, or equivalently in this circumstance, restricted) oriented directed space structure of the usual construction of untwisted finite-gauge group Dijkgraaf-Witten theory.  As for source defects in \cite{DPY, LY} it is possible to identify the appropriate analogue of a group cocycle on a $\Gamma_2$-parcel to provide the twisted theory.

The coefficients for the twisted theory will be an assignment to each colored top-dimensional simplex to an invertible scalar (usually taken as a complex number, though theories with coefficients in another field or even commutative ring may be considered), dependent upon the arrows with which its edges are colored and the relationship of the orientation of the simplex induced by the order of traversing the edges to the orientation of the bulk.  This, of course, means assigning coefficients to possible labelings of the long paths of simplexes.  As in other Dijkgraaf-Witten type state-sum constructions, the coefficient assigned when a simplex oriented by the long path is opposite that of the bulk will be the multiplicative inverse of the coefficient assigned to a simplex with the same sequence of labels on its long path, but with orientation agreeing with the bulk.

As usual, we need conditions on the coefficients $\alpha(x_1,\ldots x_d)$ so that the sum over colorings of their product over top dimensional simplexes (with the factor inverted when the simplex has the reverse orientation) will be invariant under Pachner moves, and the extended Pachner moves of \cite{CY}.  Letting $\Lambda({\mathcal T}, {\mathcal C})$ denote the set of $\mathcal C$-colorings of a flag-like triangulation $\mathcal T$, $\lambda(\sigma)$ denote the coloring of the long path of a simplex $\sigma$ by $\lambda \in \Lambda({\mathcal T}, {\mathcal C})$, and $\epsilon(\sigma)$ the orientation sign of $\sigma$, once suitable cocycle-like conditions have been required of $\alpha$, the quantity
\[  \# \gamma^{-1}(Id_\beta)^{\sigma_\beta - v_\beta}\cdot
(\# \gamma^{-1}(Id_\delta)^{\sigma_\delta - v_\delta} \cdot \sum_{\lambda \in \Lambda({\mathcal T}, {\mathcal C})} \prod_{\sigma \in {\mathcal T}_d} \alpha(\lambda(\sigma))^{\epsilon(\sigma)} \;\;\; (*)\]

\noindent will be an invariant of the pair of oriented manifolds $X_{d-1} \subset X_d$, arising from a TQFT with defects (cf. \cite{Y2}).

We will denote the arrow coloring the $i^{th}$ edge of a sequence of edges representing a factorization of a directed path lying in the 1-skeleton, by $a_i$ if the edge lies entriely in the bulk, $b_i$ if it lies entirely in the defect, $c_i$ if it has an initial vertex in the defect but otherwise lies in the bulk, and $d_i$ if it has a terminal vertex in the defect, but otherwise lies in the bulk.

Thus the $a_i$ will all lie above the identity arrow on $\beta$, the $b_i$ will all lie above the identity arrow on $\gamma$, the $c_i$ will lie above the generating arrow coming from the edge $(\gamma, \beta)$ and the $d_i$ above the generating arrow coming from the edge $(\beta, \gamma)$.

It is easy to see that top dimensional simplexes will have long paths of one of the forms

\begin{itemize}
\item $(a_1, a_2, \ldots, a_d),$
\item $(b_1, \ldots b_{k-1}, c_k, a_{k+1}, \ldots a_d)$ or
\item $(a_1, \ldots a_{k-1}, d_k, b_{k+1}, \ldots b_d),$
\end{itemize} 

\noindent in the latter two cases with $k$ ranging from 1 to $d$ (there being either no $a$'s or no $b$'s in the instances where $k=1$ or $k=d$).

Consider then a $1$ to $d+1$ Pachner move (that is a an Alexander subdivision of a top-dimensional simplex).  The new vertex introduced lies in the bulk, and as usual, the requirement that every coloring of the finer triangulation agreed with the old triangulation on edges common to both give the same value as the old triangulation, gives a cocycle condition, solved for the factor corresponding to the simplex which had beens subdivided:

For each of the three forms of top dimensional simplexes giving

\begin{itemize}
\item $\alpha(a_1, a_2, \ldots, a_d) = \alpha(a_0a_1,a_2 \ldots a_d)\alpha(a_0,a_1a_2,\ldots a_d)^{-1} \ldots$

$ \alpha(a_0,a_1,\ldots a_{d-1}a_d)^{(-1)^{d-1}}\alpha(a_0,a_1,\ldots a_{d-1})^{(-1)^d},$

\item $\alpha(b_1, \ldots b_{k-1}, c_k, a_{k+1}, \ldots a_d) = \alpha(b_1, \ldots b_{k-1}, c_k, a_{k+1}, \ldots a_da_{d+1}) \ldots$ 

$\alpha(b_1, \ldots b_{k-1}, c_k, a_{k+1}a_{k+2}, \ldots a_{d+1})^{{-1}^{d-k-1}}$ 

$\alpha(b_1, \ldots b_{k-1}, c_ka_{k+1}, \ldots a_{d+1})^{(-1)^{d-k}}$ 

$\alpha(b_1, \ldots b_{k-2}, b_{k-1}c_k, a_{k+1}, \ldots a_{d+1})^{(-1)^{d-k+1}}$

$\alpha(b_1, \ldots b_{k-2}b_{k-1}, c_k, a_{k+1}, \ldots a_{d+1})^{(-1)^{d-k+2}} \ldots$

$\alpha(b_1b_2, \ldots b_{k-1}, c_k, a_{k+1}, \ldots a_{d+1})^{(-1)^{d-1}} $

$\alpha(b_2, \ldots b_{k-1}, c_k, a_{k+1}, \ldots a_{d+1})^{(-1)^{d}} $ and

\item $\alpha(a_1, \ldots a_{k-1}, d_k, b_{k+1}, \ldots b_d) = \alpha(a_0a_1, \ldots a_{k-1}, d_k, b_{k+1}, \ldots b_d)\ldots$

$ \alpha(a_0, a_1, \ldots a_{k-2}a_{k-1}, d_k, b_{k+1}, \ldots b_d)^{(-1)^{k-2}}$

$\alpha(a_0, \ldots a_{k-2}, a_{k-1}d_k, b_{k+1}, \ldots b_d)^{(-1)^{k-1}}$

$\alpha(a_0, \ldots a_{k-1}, d_kb_{k+1}, b_{k+2} \ldots b_d)^{(-1)^k} \ldots$

$ \alpha(a_0, \ldots a_{k-1}, d_k, b_{k+1}, \ldots b_{d-1}b_d)^{(-1)^{d-1}}$ 

$\alpha(a_0, \ldots a_{k-1}, d_k, b_{k+1}, \ldots b_{d-1})^{(-1)^d},$
\end{itemize}

Observe that the extra vertex forming the subdivision was added first in the first and third cases and last in the second case.  In each instance, solving to move all the terms to one side gives the obvious generalization of a group $d$-cocycle to composable $d$-tuples of arrows in the parcel with the restriction that the arrows lie above generators or identity arrow in the path category on $\Gamma_2$, and at most one lies above a non-identity arrow.

It is easy to see that solving to put $k$ factors  on one side with the remaining $d+2-k$ on the other gives the condition needed to ensure that for any particular coloring the value before a $k$ to $d+2-k$ Pachner move is the same as that for the coloring induced on the final state after the move, and thus to ensure invariance of the state-sum (*) under Pachner moves in the bulk.

It then remains to determine the conditions to ensure invariance under extended Pachner moves which change the triangulation of the defect.  In the present setting, these moves are simply the supension of a $d-1$ dimensional Pachner move in the defect.  Each $d-1$ simplex of the defect involved in such a move is the shared face of two $d$-simplexes, one with the edges not lying in the defect oriented from the bulk to the defect, the other with edges not lying in the defect oriented from the defect to the bulk.  The long paths of these being of the form $(d_1, b_2, \ldots b_d)$ and $(b_1,\ldots b_{d-1} c_d)$, respectively.

In general a Pachner move in the defect replaces one cell of a decomposition of the boundary of a $d$-simplex into two triangulated cells with the other cell of the decomposition.  If the long path of this $d$-simplex is labled $b_1,\ldots b_d$, then there will be a long path through the suspension with a coloring of the form $c_0, b_1, \dots b_d, d_{d+1}$, which determines the coloring of all of the edges in the befor and after states of the extended Pachner move.  A $k$ to $d+1-k$ Pachner move in the defect induces a $2k$ to $2d + 2 -2k$ extended Pachner move on the triangulation of the pair.  It is fairly easy to verify that the equations required for invariance under extended Pachner moves will all result from solving the condition
\medskip

$\alpha(c_0b_1, b_2\ldots b_d)\alpha(b_2, b_3, \dots b_d, d_{d+1})\cdot $

$\alpha(c_0, b_1b_2, b_3\ldots b_d)^{-1}\alpha(b_1b_2, b_3, \dots b_d, d_{d+1})^{-1}\cdot \ldots $

$\alpha(c_0, b_1, b_2,\ldots b_\ell b_{\ell+1}\ldots b_d)^{(-1)^\ell}\alpha(b_1, b_2,\ldots b_\ell b_{\ell+1}\ldots b_d, d_{d+1})^{(-1)^{\ell}}\cdot \ldots $

$\alpha(c_0, b_1, b_2,\ldots b_{d-1}b_d)^{(-1)^{d-1}}\alpha(b_1, b_2,\ldots b_{d-1}b_d, d_{d+1})^{(-1)^{d-1}}\cdot $

$\alpha(c_0, b_1, b_2,\ldots  b_{d-1})^{(-1)^d}\alpha(b_1, b_2,\ldots b_{d-1}, b_dd_{d+1})^{(-1)^{d}} = 1$
\medskip

We thus can define the appropriate coefficients for an analogue of twisted finite gauge-group Dijkgraaf-Witten theory for oriented manifolds with defect:

\begin{definition}
A {\em partial $d$-cocycle} on a $\Gamma_2$-parcel $\mathcal C$ is a function $\alpha$ from the set of composable $d$-tuples of arrows in $\mathcal C$ all lying above either and identity arrow or a generating edge of $\mathcal P(\Gamma_2)$, with at most one arrow lying above a non-identity arrow to the group of units of a field $K$, such that 

\begin{enumerate}
\item For every composable $d+1$-tuple with at most one edge lying above a non-identity arrow, 
$x_0,\ldots x_d$, 
\smallskip

$\delta(\alpha)(x_0,\ldots x_d) = \alpha(x_1,\ldots x_d) \alpha(x_0x_1, x_2, \ldots x_d)^{-1} \ldots$

$\ldots \alpha(x_0,\ldots x_{d-1}x_d)^{(-1)^d} \alpha(x_0, \ldots x_{d-1}) = 1$
\smallskip

\item For every composable $d+2$-tuple of the form $c_0, b_1, \dots b_d, d_{d+1}$, the condition of the previous paragraph holds.
\end{enumerate}
\end{definition}

We thus have 

\begin{theorem}  If $\mathcal C$ is a regular fiber-finite $\Gamma_2$-parcel and $\alpha$ is a partial $d$-cocycle on it, for any flag-like triangulation of an oriented manifold with defect $\mathcal T$ the quantity
\[  \# \gamma^{-1}(Id_\beta)^{\sigma_\beta - v_\beta}\cdot
(\# \gamma^{-1}(Id_\delta)^{\sigma_\delta - v_\delta} \cdot \sum_{\lambda \in \Lambda({\mathcal T}, {\mathcal C})} \prod_{\sigma \in {\mathcal T}_d} \alpha(\lambda(\sigma))^{\epsilon(\sigma)} \;\;\; (*)\]

\noindent is independent of the triangulation, and thus an invariant of the oriented manifold with defect.
\end{theorem}

Over the next few pages, we illustrate with figures the derivation of the relation from the extended Pachner moves in the case $d = 3$.
The general case is entirely similar.
\smallskip

\begin{tabular}{ll}

$\displaystyle 
\begin{tikzpicture}
\fill[cyan,cyan, opacity=.7] (0,0)--(1.9,-0.5)--(3,0.5)--(0,0);
	\node[circle, fill, inner sep=.9pt, outer sep=0pt] (A) at (0,0){};
	\node[circle, fill, inner sep=.9pt, outer sep=0pt] (B) at (1.9,-0.5){};
	\node[circle, fill, inner sep=.9pt, outer sep=0pt] (C) at (3,0.5){};
	\node[circle, fill, inner sep=.9pt, outer sep=0pt] (D) at (1.5,1.8){};
	\node[circle, fill, inner sep=.9pt, outer sep=0pt] (E) at (1.5,-1.8){};
		\begin{scope}
			\draw [midarrow={>}] (A)--(B) node[font=\tiny, midway, below]{$ b_1 $};
			\draw [midarrow={>}] (B)--(C) node[font=\tiny, midway, left]{$ b_2 $};
			\draw [dashed,midarrow={>}] (A)--(C) node[font=\tiny, midway, below]{$ b_1 b_2 $\,\,\,};	
			\draw [midarrow={>}] (D)--(A) node[font=\tiny, midway, left]{$ c_0 $};
			\draw [midarrow={>}] (D)--(B) node[font=\tiny, midway, right]{$ c_0 b_1 $};
			\draw [midarrow={>}] (D)--(C) node[font=\tiny, midway, right]{$ c_0 b_1 b_2$};
			\draw [midarrow={>}] (A)--(E) node[font=\tiny, midway, left]{$ b_1 b_2 b_3 d_4 $};
			\draw [midarrow={>}] (B)--(E) node[font=\tiny, midway, left] {$ b_2 b_3 d_4 $};
			\draw [midarrow={>}] (C)--(E) node[font=\tiny, midway, right]{$ b_3 d_4 $};
			\fill[blue, opacity=.7] (A) circle (3pt);
			\fill[blue, opacity=.7] (B) circle (3pt);
			\fill[blue, opacity=.7] (C) circle (3pt);
			\fill[green, opacity=.7] (D) circle (3pt);
			\fill[green, opacity=.7] (E) circle (3pt);
			\draw[thick,blue,blue] (A)--(B);
			\draw[thick,blue,blue] (B)--(C);
			\draw[dashed,thick,blue,blue] (A)--(C);			
		\end{scope}
\end{tikzpicture}$

&
 
$ \left\{ 
 \begin{array}{ll}

\begin{tikzpicture}
\fill[cyan,cyan, opacity=.7] (0,0)--(1.9,-0.5)--(3,0.5)--(0,0);
	\node[circle, fill, inner sep=.9pt, outer sep=0pt] (A) at (0,0){};
	\node[circle, fill, inner sep=.9pt, outer sep=0pt] (B) at (1.9,-0.5){};
	\node[circle, fill, inner sep=.9pt, outer sep=0pt] (C) at (3,0.5){};
	\node[circle, fill, inner sep=.9pt, outer sep=0pt] (D) at (1.5,1.8){};
		\begin{scope}
			\draw [midarrow={>}] (A)--(B) node[font=\tiny, midway, below]{$ b_1 $};
			\draw [midarrow={>}] (B)--(C) node[font=\tiny, midway, left]{$ b_2 $};
			\draw [dashed,midarrow={>}] (A)--(C) node[font=\tiny, midway, below]{$ b_1 b_2 $\,\,\,};	
			\draw [midarrow={>}] (D)--(A) node[font=\tiny, midway, left]{$ c_0 $};
			\draw [midarrow={>}] (D)--(B) node[font=\tiny, midway, right]{$ c_0 b_1 $};
			\draw [midarrow={>}] (D)--(C) node[font=\tiny, midway, right]{$ c_0 b_1 b_2$};
			\fill[blue, opacity=.7] (A) circle (3pt);
			\fill[blue, opacity=.7] (B) circle (3pt);
			\fill[blue, opacity=.7] (C) circle (3pt);
			\fill[green, opacity=.7] (D) circle (3pt);
			\draw[thick,blue,blue] (A)--(B);
			\draw[thick,blue,blue] (B)--(C);
			\draw[dashed,thick,blue,blue] (A)--(C);			
		\end{scope}
\end{tikzpicture}

& \alpha( c_0 , b_1 , b_2 ) ^{-1} 

\\
\\

\begin{tikzpicture}
\fill[cyan,cyan, opacity=.7] (0,0)--(1.9,-0.5)--(3,0.5)--(0,0);
	\node[circle, fill, inner sep=.9pt, outer sep=0pt] (A) at (0,0){};
	\node[circle, fill, inner sep=.9pt, outer sep=0pt] (B) at (1.9,-0.5){};
	\node[circle, fill, inner sep=.9pt, outer sep=0pt] (C) at (3,0.5){};
	\node[circle, fill, inner sep=.9pt, outer sep=0pt] (E) at (1.5,-1.8){};
		\begin{scope}
			\draw [midarrow={>}] (A)--(B) node[font=\tiny, midway, below]{$ b_1 $};
			\draw [midarrow={>}] (B)--(C) node[font=\tiny, midway, left]{$ b_2 $};
			\draw [midarrow={>}] (A)--(C) node[font=\tiny, midway, above]{$ b_1 b_2 $\,\,\,};	
			\draw [midarrow={>}] (A)--(E) node[font=\tiny, midway, left]{$ b_1 b_2 b_3 d_4 $};
			\draw [midarrow={>}] (B)--(E) node[font=\tiny, midway, left] {$ b_2 b_3 d_4 $};
			\draw [midarrow={>}] (C)--(E) node[font=\tiny, midway, right]{$ b_3 d_4 $};
			\fill[blue, opacity=.7] (A) circle (3pt);
			\fill[blue, opacity=.7] (B) circle (3pt);
			\fill[blue, opacity=.7] (C) circle (3pt);
			\fill[green, opacity=.7] (E) circle (3pt);
			\draw[thick,blue,blue] (A)--(B);
			\draw[thick,blue,blue] (B)--(C);
			\draw[thick,blue,blue] (A)--(C);			
		\end{scope}
\end{tikzpicture}

& \displaystyle \alpha( b_1 , b_2 , b_3 d_4 )^{-1} 

\\ 
\end{array} \right.$

\\
\\

\end{tabular}

\begin{tabular}{ll}

$\displaystyle 
\begin{tikzpicture}
\fill[cyan,cyan, opacity=.7] (0,0)--(1.9,-0.5)--(3,0.5)--(0,0);
	\node[circle, fill, inner sep=.9pt, outer sep=0pt] (A) at (0,0){};
	\node[circle, fill, inner sep=.9pt, outer sep=0pt] (B) at (1.9,-0.5){};
	\node[circle, fill, inner sep=.9pt, outer sep=0pt] (C) at (3,0.5){};
	\node[circle, fill, inner sep=.9pt, outer sep=0pt] (D) at (1.5,1.8){};
	\node[circle, fill, inner sep=.9pt, outer sep=0pt] (E) at (1.5,-1.8){};
	\node[circle, fill, inner sep=.9pt, outer sep=0pt] (F) at (1.5,-0.1){};
		\begin{scope}
			\draw [midarrow={>}] (A)--(B) node[font=\tiny, midway, below]{$ b_1 $};
			\draw [midarrow={>}] (B)--(C) node[font=\tiny, midway, below]{\!$ b_2 $};
			\draw [dashed,midarrow={>}] (A)--(C) node[font=\tiny, midway, above]{$ b_1 b_2 $\,\,\,\,\,\,\,\,\,\,\,\,\,};	
			\draw [midarrow={>}] (D)--(A) node[font=\tiny, midway, left]{$ c_0 $};
			\draw [midarrow={>}] (D)--(B) node[font=\tiny, midway, right]{$ c_0 b_1 $};
			\draw [midarrow={>}] (D)--(C) node[font=\tiny, midway, right]{$ c_0 b_1 b_2 $};
			\draw [midarrow={>}] (A)--(E) node[font=\tiny, midway, left]{$ b_1 b_2 b_3 d_4 $};
			\draw [midarrow={>}] (B)--(E) node[font=\tiny, midway, right] {$ b_2 b_3 d_4 $};
			\draw [midarrow={>}] (C)--(E) node[font=\tiny, midway, right]{$ b_3 d_4 $};
			\draw [dashed,thick, midarrow={>}] (1.5,1.8) -- (1.5,-0.1) node[font=\tiny, midway, right]{$ c_0 b_1 b_2 b_3 $};
			\draw [dashed,thick, midarrow={>}] (1.5,-0.25) -- (1.5,-1.8) node[font=\tiny, midway, left]{$ d_4 $};
			\fill[blue, opacity=.7] (A) circle (3pt);
			\fill[blue, opacity=.7] (B) circle (3pt);
			\fill[blue, opacity=.7] (C) circle (3pt);
			\fill[green, opacity=.7] (D) circle (3pt);
			\fill[green, opacity=.7] (E) circle (3pt);
			\draw[thick,blue,blue] (A)--(B);
			\draw[thick,blue,blue] (B)--(C);
			\draw[dashed,thick,blue,blue] (A)--(C);	
			\draw [thick, midarrow={>}] (A)--(F) node[font=\tiny, midway, above]{$ b_1 b_2 b_3 $};
			\draw [thick, midarrow={>}] (B)--(F) node[font=\tiny, midway, left]{$ b_2 b_3 $};
			\draw [thick, midarrow={>}] (C)--(F) node[font=\tiny, midway, below]{\!\!\!\!\!$ b_3 $};		
		\end{scope}
\end{tikzpicture}$

&
 
$ \left\{ 
 \begin{array}{ll}

\begin{tikzpicture}
\fill[cyan,cyan, opacity=.7] (0,0)--(1.9,-0.5)--(1.5,0.1)--(0,0);
	\node[circle, fill, inner sep=.9pt, outer sep=0pt] (A) at (0,0){};
	\node[circle, fill, inner sep=.9pt, outer sep=0pt] (B) at (1.9,-0.5){};
	\node[circle, fill, inner sep=.9pt, outer sep=0pt] (D) at (1.5,1.8){};
	\node[circle, fill, inner sep=.9pt, outer sep=0pt] (F) at (1.5,0.1){};
		\begin{scope}
			\draw [midarrow={>}] (A)--(B) node[font=\tiny, midway, below]{$ b_1 $};
			\draw [midarrow={>}] (D)--(A) node[font=\tiny, midway, left]{$ c_0 $};
			\draw [midarrow={>}] (D)--(B) node[font=\tiny, midway, right]{$ c_0 b_1 $};
			\draw [dashed,thick, midarrow={>}] (1.5,1.8) -- (1.5,0.1) node[font=\tiny, midway, right]{\,\,$ c_0 b_1 b_2 b_3 $};
			\fill[blue, opacity=.7] (A) circle (3pt);
			\fill[blue, opacity=.7] (B) circle (3pt);
			\fill[blue, opacity=.7] (F) circle (3pt);
			\fill[green, opacity=.7] (D) circle (3pt);
			\draw[thick,blue,blue] (A)--(B);
			\draw [dashed,thick, midarrow={>}] (A)--(F) node[font=\tiny, midway, above]{$ b_1 b_2 b_3 $};
			\draw [dashed,thick, midarrow={>}] (B)--(F) node[font=\tiny, midway, right]{\,\,$ b_2 b_3 $};
			\draw[dashed,thick,blue,blue] (A)--(F);
			\draw[dashed,thick,blue,blue] (F)--(B);
		\end{scope}
\end{tikzpicture}

& \alpha( c_0 , b_1 , b_2 b_3 )^{-1} 

\\
\\

\begin{tikzpicture}
\fill[cyan,cyan, opacity=.7] (0,0)--(1.5,-0.1)--(3,0.5)--(0,0);
	\node[circle, fill, inner sep=.9pt, outer sep=0pt] (A) at (0,0){};
	\node[circle, fill, inner sep=.9pt, outer sep=0pt] (C) at (3,0.5){};
	\node[circle, fill, inner sep=.9pt, outer sep=0pt] (D) at (1.5,1.8){};
	\node[circle, fill, inner sep=.9pt, outer sep=0pt] (F) at (1.5,-0.1){};
		\begin{scope}
			\draw [dashed,midarrow={>}] (A)--(C) node[font=\tiny, midway, above]{$ b_1 b_2 $\,\,\,\,\,\,\,\,\,\,\,\,\,};	
			\draw [midarrow={>}] (D)--(A) node[font=\tiny, midway, left]{$ c_0 $};
			\draw [midarrow={>}] (D)--(C) node[font=\tiny, midway, right]{$ c_0 b_1 b_2 $};
			\draw [thick, midarrow={>}] (1.5,1.8) -- (1.5,-0.1) node[font=\tiny, midway, right]{\!\!$ c_0 b_1 b_2 b_3 $};
			\fill[blue, opacity=.7] (A) circle (3pt);
			\fill[blue, opacity=.7] (C) circle (3pt);
			\fill[blue, opacity=.7] (F) circle (3pt);
			\fill[green, opacity=.7] (D) circle (3pt);
			\draw[dashed,thick,blue,blue] (A)--(C);	
			\draw [thick, midarrow={>}] (A)--(F) node[font=\tiny, midway, below]{$ b_1 b_2 b_3 $};
			\draw [thick, midarrow={>}] (C)--(F) node[font=\tiny, midway, below]{$ b_3 $};
			\draw[thick,blue,blue] (F)--(A);
			\draw[thick,blue,blue] (F)--(C);		
		\end{scope}
\end{tikzpicture}

& \displaystyle \alpha ( c_0 , b_1 b_2 , b_3 )

\\ 
\\

\begin{tikzpicture}
\fill[cyan,cyan, opacity=.7] (1.5,-0.1)--(1.9,-0.5)--(3,0.5)--(1.5,-0.1);
	\node[circle, fill, inner sep=.9pt, outer sep=0pt] (B) at (1.9,-0.5){};
	\node[circle, fill, inner sep=.9pt, outer sep=0pt] (C) at (3,0.5){};
	\node[circle, fill, inner sep=.9pt, outer sep=0pt] (D) at (1.5,1.8){};
	\node[circle, fill, inner sep=.9pt, outer sep=0pt] (F) at (1.5,-0.1){};
		\begin{scope}
			\draw [midarrow={>}] (B)--(C) node[font=\tiny, midway, below]{\!$ b_2 $};
			\draw [midarrow={>}] (D)--(B) node[font=\tiny, midway, right]{$ c_0 b_1 $};
			\draw [midarrow={>}] (D)--(C) node[font=\tiny, midway, right]{$ c_0 b_1 b_2 $};
			\draw [thick, midarrow={>}] (1.5,1.8) -- (1.5,-0.1) node[font=\tiny, midway, left]{$ c_0 b_1 b_2 b_3 $};
			\fill[blue, opacity=.7] (B) circle (3pt);
			\fill[blue, opacity=.7] (C) circle (3pt);
			\fill[blue, opacity=.7] (F) circle (3pt);
			\fill[green, opacity=.7] (D) circle (3pt);
			\draw[thick,blue,blue] (B)--(C);
			\draw [thick, midarrow={>}] (B)--(F) node[font=\tiny, midway, below]{$ b_2 b_3 $\,\,\,\,\,\,\,\,};
			\draw [dashed,thick, midarrow={>}] (C)--(F) node[font=\tiny, midway, below]{\!\!\!\!\!\!\!\!$ b_3 $};
			\draw[thick,blue,blue] (F)--(B);
			\draw[dashed,thick,blue,blue] (F)--(C);			
		\end{scope}
\end{tikzpicture}

&  \alpha( c_0 b_1 , b_2 , b_3 )^{-1} 

\\
\\

\begin{tikzpicture}
\fill[cyan,cyan, opacity=.7] (0,0)--(1.9,-0.5)--(1.5,0.1)--(0,0);
	\node[circle, fill, inner sep=.9pt, outer sep=0pt] (A) at (0,0){};
	\node[circle, fill, inner sep=.9pt, outer sep=0pt] (B) at (1.9,-0.5){};
	\node[circle, fill, inner sep=.9pt, outer sep=0pt] (E) at (1.5,-1.8){};
	\node[circle, fill, inner sep=.9pt, outer sep=0pt] (F) at (1.5,0.1){};
		\begin{scope}
			\draw [midarrow={>}] (A)--(B) node[font=\tiny, midway, below]{$ b_1 $};
			\draw [midarrow={>}] (A)--(E) node[font=\tiny, midway, left]{$ b_1 b_2 b_3 d_4 $};
			\draw [midarrow={>}] (B)--(E) node[font=\tiny, midway, right] {$ b_2 b_3 d_4 $};
			\draw [dashed,thick, midarrow={>}] (F) -- (1.5,-1.8) node[font=\tiny, midway, left]{$ d_4 $};
			\fill[blue, opacity=.7] (A) circle (3pt);
			\fill[blue, opacity=.7] (B) circle (3pt);
			\fill[blue, opacity=.7] (F) circle (3pt);
			\fill[green, opacity=.7] (E) circle (3pt);
			\draw[thick,blue,blue] (A)--(B);
			\draw [thick, midarrow={>}] (A)--(F) node[font=\tiny, midway, above]{$ b_1 b_2 b_3 $};
			\draw [thick, midarrow={>}] (B)--(F) node[font=\tiny, midway, right]{$ b_2 b_3 $};
			\draw[thick,blue,blue] (F)--(B);
			\draw[thick,blue,blue] (F)--(A);	
		\end{scope}
\end{tikzpicture}

&  \alpha( b_1 , b_2 b_3 , d_4 )^{-1} 

\\
\\

\begin{tikzpicture}
\fill[cyan,cyan, opacity=.7] (0,0)--(1.5,-0.1)--(3,0.5)--(0,0);
	\node[circle, fill, inner sep=.9pt, outer sep=0pt] (A) at (0,0){};
	\node[circle, fill, inner sep=.9pt, outer sep=0pt] (C) at (3,0.5){};
	\node[circle, fill, inner sep=.9pt, outer sep=0pt] (E) at (1.5,-1.8){};
	\node[circle, fill, inner sep=.9pt, outer sep=0pt] (F) at (1.5,-0.1){};
		\begin{scope}
			\draw [midarrow={>}] (A)--(C) node[font=\tiny, midway, above]{$ b_1 b_2 $\,\,\,\,\,\,\,};	
			\draw [midarrow={>}] (A)--(E) node[font=\tiny, midway, left]{$ b_1 b_2 b_3 d_4 $};
			\draw [midarrow={>}] (C)--(E) node[font=\tiny, midway, right]{$ b_3 d_4 $};
			\draw [thick, midarrow={>}] (F) -- (1.5,-1.8) node[font=\tiny, midway, left]{$ d_4 $};
			\fill[blue, opacity=.7] (A) circle (3pt);
			\fill[blue, opacity=.7] (C) circle (3pt);
			\fill[blue, opacity=.7] (F) circle (3pt);
			\fill[green, opacity=.7] (E) circle (3pt);
			\draw[dashed,thick,blue,blue] (A)--(C);	
			\draw [thick, midarrow={>}] (A)--(F) node[font=\tiny, midway, below]{$ b_1 b_2 b_3 $};
			\draw [thick, midarrow={>}] (C)--(F) node[font=\tiny, midway, below]{\!\!\!\!\!$ b_3 $};	
			\draw[thick,blue,blue] (F)--(C);
			\draw[thick,blue,blue] (F)--(A);		
		\end{scope}
\end{tikzpicture}

&  \alpha ( b_1 b_2 , b_3 , d_4 )

\\
\\

\begin{tikzpicture}
\fill[cyan,cyan, opacity=.7] (1.5,-0.1)--(1.9,-0.5)--(3,0.5)--(1.5,-0.1);
	\node[circle, fill, inner sep=.9pt, outer sep=0pt] (B) at (1.9,-0.5){};
	\node[circle, fill, inner sep=.9pt, outer sep=0pt] (C) at (3,0.5){};
	\node[circle, fill, inner sep=.9pt, outer sep=0pt] (E) at (1.5,-1.8){};
	\node[circle, fill, inner sep=.9pt, outer sep=0pt] (F) at (1.5,-0.1){};
		\begin{scope}
			\draw [midarrow={>}] (B)--(C) node[font=\tiny, midway, below]{\!$ b_2 $};
			\draw [midarrow={>}] (B)--(E) node[font=\tiny, midway, right] {$ b_2 b_3 d_4 $};
			\draw [midarrow={>}] (C)--(E) node[font=\tiny, midway, right]{$ b_3 d_4 $};
			\draw [thick, midarrow={>}] (F) -- (1.5,-1.8) node[font=\tiny, midway, left]{$ d_4 $};
			\fill[blue, opacity=.7] (B) circle (3pt);
			\fill[blue, opacity=.7] (C) circle (3pt);
			\fill[blue, opacity=.7] (F) circle (3pt);
			\fill[green, opacity=.7] (E) circle (3pt);
			\draw[thick,blue,blue] (B)--(C);
			\draw [thick, midarrow={>}] (B)--(F) node[font=\tiny, midway, above]{\,\,\,\,\,\,\,\,\,\,\,$ b_2 b_3 $};
			\draw [thick, midarrow={>}] (C)--(F) node[font=\tiny, midway, above]{\!\!\!\!\!$ b_3 $};	
			\draw[thick,blue,blue] (F)--(C);
			\draw[thick,blue,blue] (F)--(B);		
		\end{scope}
\end{tikzpicture}

&  \alpha( b_2 , b_3 , d_4 )^{-1} 

\\

\end{array} \right.$

\\
\\

\end{tabular}

\vspace{6mm}

Thus

\begin{center}
\begin{tabular}{llll}
$ \alpha (c_0 b_1 , b_2 , b_3 )$ & $\cdot$ & $\alpha ( b_2 , b_3 , d_4 ) $ & $\cdot$\\
$ \alpha( c_0 , b_1 b_2 , b_3 )^{-1} $ & $\cdot$ & $\alpha( b_1 b_2 , b_3 , d_4 ) ^{-1} $ & $\cdot$\\
$ \alpha (c_0 , b_1 , b_2 b_3 )$ & $\cdot$ & $\alpha ( b_1 , b_2 b_3 , d_4 ) $ & $\cdot$\\
$ \alpha( c_0 , b_1 , b_2 )^{-1} $ & $\cdot$ & $\alpha( b_1 , b_2 , b_3 d_4 )^{-1}  $ & $ = 1 .$\\
\\
\end{tabular}
\end{center}


\begin{tabular}{ll}

$\displaystyle
\begin{tikzpicture}
	\fill[cyan,cyan, opacity=.7] (0,0)--(2.5,0)--(3.5,1.2)--(1,1.2)--(0,0);
 	\node[circle, fill, inner sep=.8pt, outer sep=0pt] (A) at (0,0){};
	\node[circle, fill, inner sep=.8pt, outer sep=0pt] (B) at (2.5,0){};
	\node[circle, fill, inner sep=.8pt, outer sep=0pt] (C) at (3.5,1.2){};
	\node[circle, fill, inner sep=.8pt, outer sep=0pt] (D) at (1,1.2){};
	\node[circle, fill, inner sep=.8pt, outer sep=0pt] (E) at (1.7,3){};
	\node[circle, fill, inner sep=.8pt, outer sep=0pt] (F) at (1.7,-2){};
	\begin{scope}
		\draw [midarrow={>}] (A)--(B) node[font=\tiny, midway, below]{$ b_1 $\,\,\,};
		\draw [midarrow={>}] (B)--(C) node[font=\tiny, midway, left]{$ b_2 $};
		\draw [dashed,midarrow={>}] (C)--(D) node[font=\tiny, midway, below]{\,\,\,\,$ b_3 $};
		\draw [dashed,midarrow={>}] (A)--(D) node[font=\tiny, midway, right]{$ b_1 b_2 b_3 $};
		\draw [thick,midarrow={>}] (B)--(D) node[font=\tiny, midway, right]{$ b_2 b_3 $};
		\draw [midarrow={>}] (E)--(A) node[font=\tiny, midway, left]{$ c_0 $};
		\draw [midarrow={>}] (E)--(B) node[font=\tiny, midway, right]{$ c_0 b_1 $};
		\draw [midarrow={>}] (E)--(C) node[font=\tiny, midway, right]{$ c_0 b_1 b_2 $};
		\draw [dashed,midarrow={>}] (E)--(D) node[font=\tiny, midway, right]{$ c_0 b_1 b_2 b_3 $};
		\draw [midarrow={>}] (A)--(F) node[font=\tiny, midway, left]{$ b_1 b_2 b_3 d_4  $};
		\draw [midarrow={>}] (B)--(F) node[font=\tiny, midway, left]{$ b_2 b_3 d_4$\!\!};
		\draw [midarrow={>}] (C)--(F) node[font=\tiny, midway, right]{$ b_3 d_4 $};
		\draw [dashed,midarrow={>}] (D)--(F) node[font=\tiny, midway, right]{$ d_4 $};
		\draw[thick,blue,blue] (A)--(B)--(C);
		\draw[dashed,thick,blue,blue] (C)--(D)--(A);
		\fill[blue, opacity=.7] (A) circle (3pt);
		\fill[blue, opacity=.7] (B) circle (3pt);
		\fill[blue, opacity=.7] (C) circle (3pt);
		\fill[blue, opacity=.7] (D) circle (3pt);
		\fill[green, opacity=.7] (E) circle (3pt);
		\fill[green, opacity=.7] (F) circle (3pt);
	\end{scope}
\end{tikzpicture}$

&
 
$ \left\{ 
 \begin{array}{ll}

\begin{tikzpicture}
	\fill[cyan,cyan, opacity=.7] (0,0)--(2.5,0)--(1,1.2)--(0,0);
 	\node[circle, fill, inner sep=.8pt, outer sep=0pt] (A) at (0,0){};
	\node[circle, fill, inner sep=.8pt, outer sep=0pt] (B) at (2.5,0){};
	\node[circle, fill, inner sep=.8pt, outer sep=0pt] (D) at (1,1.2){};
	\node[circle, fill, inner sep=.8pt, outer sep=0pt] (E) at (1.7,3){};
	\begin{scope}
		\draw [midarrow={>}] (A)--(B) node[font=\tiny, midway, below]{$ b_1 $\,\,\,};
		\draw [dashed,midarrow={>}] (A)--(D) node[font=\tiny, midway, right]{$ b_1 b_2 b_3 $};
		\draw [dashed,thick,midarrow={>}] (B)--(D) node[font=\tiny, midway, right]{$ b_2 b_3 $};
		\draw [midarrow={>}] (E)--(A) node[font=\tiny, midway, left]{$ c_0 $};
		\draw [midarrow={>}] (E)--(B) node[font=\tiny, midway, right]{$ c_0 b_1 $};
		\draw [dashed,midarrow={>}] (E)--(D) node[font=\tiny, midway, right]{$ c_0 b_1 b_2 b_3 $};
		\draw[thick,blue,blue] (A)--(B);
		\draw[dashed,thick,blue,blue] (D)--(A);
		\draw[dashed,thick,blue,blue] (D)--(B);
		\fill[blue, opacity=.7] (A) circle (3pt);
		\fill[blue, opacity=.7] (B) circle (3pt);
		\fill[blue, opacity=.7] (D) circle (3pt);
		\fill[green, opacity=.7] (E) circle (3pt);
	\end{scope}
\end{tikzpicture}

& \alpha( c_0 , b_1 , b_2 b_3 )^{-1} 
\\
\\

\begin{tikzpicture}
	\fill[cyan,cyan, opacity=.7] (2.5,0)--(3.5,1.2)--(1,1.2);
	\node[circle, fill, inner sep=.8pt, outer sep=0pt] (B) at (2.5,0){};
	\node[circle, fill, inner sep=.8pt, outer sep=0pt] (C) at (3.5,1.2){};
	\node[circle, fill, inner sep=.8pt, outer sep=0pt] (D) at (1,1.2){};
	\node[circle, fill, inner sep=.8pt, outer sep=0pt] (E) at (1.7,3){};
	\begin{scope}
		\draw [midarrow={>}] (B)--(C) node[font=\tiny, midway, left]{$ b_2 $};
		\draw [dashed,midarrow={>}] (C)--(D) node[font=\tiny, midway, below]{\,\,\,\,$ b_3 $};
		\draw [thick,midarrow={>}] (B)--(D) node[font=\tiny, midway, left]{$ b_2 b_3 $};
		\draw [midarrow={>}] (E)--(B) node[font=\tiny, midway, left]{$ c_0 b_1 $};
		\draw [midarrow={>}] (E)--(C) node[font=\tiny, midway, right]{$ c_0 b_1 b_2 $};
		\draw [midarrow={>}] (E)--(D) node[font=\tiny, midway, left]{$ c_0 b_1 b_2 b_3 $};
		\draw[thick,blue,blue] (B)--(C);
		\draw[thick,blue,blue] (B)--(D);
		\draw[dashed,thick,blue,blue] (C)--(D);
		\fill[blue, opacity=.7] (B) circle (3pt);
		\fill[blue, opacity=.7] (C) circle (3pt);
		\fill[blue, opacity=.7] (D) circle (3pt);
		\fill[green, opacity=.7] (E) circle (3pt);
	\end{scope}
\end{tikzpicture}

& \displaystyle \alpha( c_0 b_1 , b_2 , b_3 )^{-1} 
\\
\\

\begin{tikzpicture}
	\fill[cyan,cyan, opacity=.7] (0,0)--(2.5,0)--(1,1.2)--(0,0);
 	\node[circle, fill, inner sep=.8pt, outer sep=0pt] (A) at (0,0){};
	\node[circle, fill, inner sep=.8pt, outer sep=0pt] (B) at (2.5,0){};
	\node[circle, fill, inner sep=.8pt, outer sep=0pt] (D) at (1,1.2){};
	\node[circle, fill, inner sep=.8pt, outer sep=0pt] (F) at (1.7,-2){};
	\begin{scope}
		\draw [midarrow={>}] (A)--(B) node[font=\tiny, midway, above]{$ b_1 $\,\,\,\,\,\,\,\,\,\,};
		\draw [midarrow={>}] (A)--(D) node[font=\tiny, midway, left]{$ b_1 b_2 b_3 $};
		\draw [thick,midarrow={>}] (B)--(D) node[font=\tiny, midway, right]{$ b_2 b_3 $};
		\draw [midarrow={>}] (A)--(F) node[font=\tiny, midway, left]{$ b_1 b_2 b_3 d_4 $};
		\draw [midarrow={>}] (B)--(F) node[font=\tiny, midway, right]{$ b_2 b_3 d_4$\!\!};
		\draw [dashed,midarrow={>}] (D)--(F) node[font=\tiny, midway, right]{$ d_4 $};
		\draw[thick,blue,blue] (A)--(B);
		\draw[thick,blue,blue] (D)--(A);
		\draw[thick,blue,blue] (D)--(B);
		\fill[blue, opacity=.7] (A) circle (3pt);
		\fill[blue, opacity=.7] (B) circle (3pt);
		\fill[blue, opacity=.7] (D) circle (3pt);
		\fill[green, opacity=.7] (F) circle (3pt);
	\end{scope}
\end{tikzpicture}

& \displaystyle \alpha ( b_1 , b_2 b_3 , d_4 )^{-1}
\\
\\

\begin{tikzpicture}
	\fill[cyan,cyan, opacity=.7] (2.5,0)--(3.5,1.2)--(1,1.2);
	\node[circle, fill, inner sep=.8pt, outer sep=0pt] (B) at (2.5,0){};
	\node[circle, fill, inner sep=.8pt, outer sep=0pt] (C) at (3.5,1.2){};
	\node[circle, fill, inner sep=.8pt, outer sep=0pt] (D) at (1,1.2){};
	\node[circle, fill, inner sep=.8pt, outer sep=0pt] (F) at (1.7,-2){};
	\begin{scope}
		\draw [midarrow={>}] (B)--(C) node[font=\tiny, midway, left]{$ b_2 $};
		\draw [midarrow={>}] (C)--(D) node[font=\tiny, midway, above]{\,\,\,\,$ b_3 $};
		\draw [thick,midarrow={>}] (B)--(D) node[font=\tiny, midway, right]{$ b_2 b_3 $};
		\draw [midarrow={>}] (B)--(F) node[font=\tiny, midway, left]{$ b_2 b_3 d_4$\!\!};
		\draw [midarrow={>}] (C)--(F) node[font=\tiny, midway, right]{$ b_3 d_4 $};
		\draw [midarrow={>}] (D)--(F) node[font=\tiny, midway, right]{$ d_4 $};
		\draw[thick,blue,blue] (B)--(C);
		\draw[thick,blue,blue] (C)--(D);
		\draw[thick,blue,blue] (B)--(D);
		\fill[blue, opacity=.7] (B) circle (3pt);
		\fill[blue, opacity=.7] (C) circle (3pt);
		\fill[blue, opacity=.7] (D) circle (3pt);
		\fill[green, opacity=.7] (F) circle (3pt);
	\end{scope}
\end{tikzpicture}

& \displaystyle \alpha( b_2 , b_3 , d_4 )^{-1} 
\\

\end{array} \right.$

\end{tabular}

\newpage

\begin{tabular}{ll}

$\displaystyle 
\begin{tikzpicture}
	\fill[cyan,cyan, opacity=.7] (0,0)--(2.5,0)--(3.5,1.2)--(1,1.2)--(0,0);
 	\node[circle, fill, inner sep=.8pt, outer sep=0pt] (A) at (0,0){};
	\node[circle, fill, inner sep=.8pt, outer sep=0pt] (B) at (2.5,0){};
	\node[circle, fill, inner sep=.8pt, outer sep=0pt] (C) at (3.5,1.2){};
	\node[circle, fill, inner sep=.8pt, outer sep=0pt] (D) at (1,1.2){};
	\node[circle, fill, inner sep=.8pt, outer sep=0pt] (E) at (1.7,3){};
	\node[circle, fill, inner sep=.8pt, outer sep=0pt] (F) at (1.7,-2){};
	\begin{scope}
		\draw [midarrow={>}] (A)--(B) node[font=\tiny, midway, below]{$ b_1 $\,\,\,};
		\draw [midarrow={>}] (B)--(C) node[font=\tiny, midway, left]{$ b_2 $};
		\draw [dashed,midarrow={>}] (C)--(D) node[font=\tiny, midway, below]{\,\,\,\,$ b_3 $};
		\draw [dashed,midarrow={>}] (A)--(D) node[font=\tiny, midway, right]{$ b_1 b_2 b_3 $};
		\draw [thick,midarrow={>}] (A)--(C) node[font=\tiny, midway, below]{$ b_1 b_2 $};
		\draw [midarrow={>}] (E)--(A) node[font=\tiny, midway, left]{$ c_0 $};
		\draw [midarrow={>}] (E)--(B) node[font=\tiny, midway, right]{$ c_0 b_1 $};
		\draw [midarrow={>}] (E)--(C) node[font=\tiny, midway, right]{$ c_0 b_1 b_2 $};
		\draw [dashed,midarrow={>}] (E)--(D) node[font=\tiny, midway, right]{$ c_0 b_1 b_2 b_3 $};
		\draw [midarrow={>}] (A)--(F) node[font=\tiny, midway, left]{$ b_1 b_2 b_3 d_4 $};
		\draw [midarrow={>}] (B)--(F) node[font=\tiny, midway, left]{$ b_2 b_3 d_4$\!\!};
		\draw [midarrow={>}] (C)--(F) node[font=\tiny, midway, right]{$ b_3 d_4 $};
		\draw [dashed,midarrow={>}] (D)--(F) node[font=\tiny, midway, right]{$ d_4 $};
		\draw[thick,blue,blue] (A)--(B)--(C);
		\draw[dashed,thick,blue,blue] (C)--(D)--(A);
		\fill[blue, opacity=.7] (A) circle (3pt);
		\fill[blue, opacity=.7] (B) circle (3pt);
		\fill[blue, opacity=.7] (C) circle (3pt);
		\fill[blue, opacity=.7] (D) circle (3pt);
		\fill[green, opacity=.7] (E) circle (3pt);
		\fill[green, opacity=.7] (F) circle (3pt);
	\end{scope}
\end{tikzpicture}$

&
 
$ \left\{ 
 \begin{array}{ll}

\begin{tikzpicture}
	\fill[cyan,cyan, opacity=.7] (0,0)--(2.5,0)--(3.5,1.2)--(0,0);
 	\node[circle, fill, inner sep=.8pt, outer sep=0pt] (A) at (0,0){};
	\node[circle, fill, inner sep=.8pt, outer sep=0pt] (B) at (2.5,0){};
	\node[circle, fill, inner sep=.8pt, outer sep=0pt] (C) at (3.5,1.2){};
	\node[circle, fill, inner sep=.8pt, outer sep=0pt] (E) at (1.7,3){};
	\begin{scope}
		\draw [midarrow={>}] (A)--(B) node[font=\tiny, midway, below]{$ b_1 $\,\,\,};
		\draw [midarrow={>}] (B)--(C) node[font=\tiny, midway, left]{$ b_2 $};
		\draw [dashed,thick,midarrow={>}] (A)--(C) node[font=\tiny, midway, below]{$ b_1 b_2 $};
		\draw [midarrow={>}] (E)--(A) node[font=\tiny, midway, left]{$ c_0 $};
		\draw [midarrow={>}] (E)--(B) node[font=\tiny, midway, right]{$ c_0 b_1 $};
		\draw [midarrow={>}] (E)--(C) node[font=\tiny, midway, right]{$ c_0 b_1 b_2 $};
		\draw[thick,blue,blue] (A)--(B)--(C);
		\draw[dashed,thick,blue,blue] (C)--(A);
		\fill[blue, opacity=.7] (A) circle (3pt);
		\fill[blue, opacity=.7] (B) circle (3pt);
		\fill[blue, opacity=.7] (C) circle (3pt);
		\fill[green, opacity=.7] (E) circle (3pt);
	\end{scope}
\end{tikzpicture}

& \alpha( c_0 , b_1 , b_2 )^{-1} 
\\
\\

\begin{tikzpicture}
	\fill[cyan,cyan, opacity=.7] (0,0)--(3.5,1.2)--(1,1.2)--(0,0);
 	\node[circle, fill, inner sep=.8pt, outer sep=0pt] (A) at (0,0){};
	\node[circle, fill, inner sep=.8pt, outer sep=0pt] (C) at (3.5,1.2){};
	\node[circle, fill, inner sep=.8pt, outer sep=0pt] (D) at (1,1.2){};
	\node[circle, fill, inner sep=.8pt, outer sep=0pt] (E) at (1.7,3){};
	\begin{scope}
		\draw [dashed,midarrow={>}] (C)--(D) node[font=\tiny, midway, below]{\,\,\,\,$ b_3 $};
		\draw [dashed,midarrow={>}] (A)--(D) node[font=\tiny, midway, left]{$ b_1 b_2 b_3 $};
		\draw [thick,midarrow={>}] (A)--(C) node[font=\tiny, midway, below]{$ b_1 b_2 $};
		\draw [midarrow={>}] (E)--(A) node[font=\tiny, midway, left]{$ c_0 $};
		\draw [midarrow={>}] (E)--(C) node[font=\tiny, midway, right]{$ c_0 b_1 b_2 $};
		\draw [dashed,midarrow={>}] (E)--(D) node[font=\tiny, midway, right]{$ c_0 b_1 b_2 b_3 $};
		\draw[thick,blue,blue] (A)--(C);
		\draw[dashed,thick,blue,blue] (C)--(D)--(A);
		\fill[blue, opacity=.7] (A) circle (3pt);
		\fill[blue, opacity=.7] (C) circle (3pt);
		\fill[blue, opacity=.7] (D) circle (3pt);
		\fill[green, opacity=.7] (E) circle (3pt);
	\end{scope}
\end{tikzpicture}

& \displaystyle \alpha( c_0 , b_1 b_2 , b_3 )^{-1} 
\\
\\

\begin{tikzpicture}
	\fill[cyan,cyan, opacity=.7] (0,0)--(2.5,0)--(3.5,1.2)--(0,0);
 	\node[circle, fill, inner sep=.8pt, outer sep=0pt] (A) at (0,0){};
	\node[circle, fill, inner sep=.8pt, outer sep=0pt] (B) at (2.5,0){};
	\node[circle, fill, inner sep=.8pt, outer sep=0pt] (C) at (3.5,1.2){};
	\node[circle, fill, inner sep=.8pt, outer sep=0pt] (F) at (1.7,-2){};
	\begin{scope}
		\draw [midarrow={>}] (A)--(B) node[font=\tiny, midway, below]{$ b_1 $\,\,\,};
		\draw [midarrow={>}] (B)--(C) node[font=\tiny, midway, left]{$ b_2 $};
		\draw [thick,midarrow={>}] (A)--(C) node[font=\tiny, midway, below]{$ b_1 b_2 $};
		\draw [midarrow={>}] (A)--(F) node[font=\tiny, midway, left]{$ b_1 b_2 b_3 d_4 $};
		\draw [midarrow={>}] (B)--(F) node[font=\tiny, midway, left]{$ b_2 b_3 d_4$\!\!};
		\draw [midarrow={>}] (C)--(F) node[font=\tiny, midway, right]{$ b_3 d_4 $};
		\draw[thick,blue,blue] (A)--(B)--(C);
		\draw[thick,blue,blue] (C)--(A);
		\fill[blue, opacity=.7] (A) circle (3pt);
		\fill[blue, opacity=.7] (B) circle (3pt);
		\fill[blue, opacity=.7] (C) circle (3pt);
		\fill[green, opacity=.7] (F) circle (3pt);
	\end{scope}
\end{tikzpicture}

& \displaystyle \alpha( b_1 , b_2 , b_3 d_4 )^{-1} 
\\
\\

\begin{tikzpicture}
	\fill[cyan,cyan, opacity=.7] (0,0)--(3.5,1.2)--(1,1.2)--(0,0);
 	\node[circle, fill, inner sep=.8pt, outer sep=0pt] (A) at (0,0){};
	\node[circle, fill, inner sep=.8pt, outer sep=0pt] (C) at (3.5,1.2){};
	\node[circle, fill, inner sep=.8pt, outer sep=0pt] (D) at (1,1.2){};
	\node[circle, fill, inner sep=.8pt, outer sep=0pt] (F) at (1.7,-2){};
	\begin{scope}
		\draw [dashed,midarrow={>}] (C)--(D) node[font=\tiny, midway, below]{\,\,\,\,$ b_3 $};
		\draw [dashed,midarrow={>}] (A)--(D) node[font=\tiny, midway, left]{$ b_1 b_2 b_3 $};
		\draw [thick,midarrow={>}] (A)--(C) node[font=\tiny, midway, below]{$ b_1 b_2 $};
		\draw [midarrow={>}] (A)--(F) node[font=\tiny, midway, left]{$ b_1 b_2 b_3 d_4 $};
		\draw [midarrow={>}] (C)--(F) node[font=\tiny, midway, right]{$ b_3 d_4 $};
		\draw [dashed,midarrow={>}] (D)--(F) node[font=\tiny, midway, right]{$ d_4 $};
		\draw[thick,blue,blue] (A)--(C);
		\draw[thick,blue,blue] (C)--(D)--(A);
		\fill[blue, opacity=.7] (A) circle (3pt);
		\fill[blue, opacity=.7] (C) circle (3pt);
		\fill[blue, opacity=.7] (D) circle (3pt);
		\fill[green, opacity=.7] (F) circle (3pt);
	\end{scope}
\end{tikzpicture}

& \displaystyle \alpha( b_1 b_2 , b_3 , d_4 )^{-1} 
\\

\end{array} \right.$

\end{tabular}

\bigskip
Thus again we must have\smallskip

\begin{center}
\begin{tabular}{llll}
$ \alpha (c_0 b_1 , b_2 , b_3 )$ & $\cdot$ & $\alpha ( b_2 , b_3 , d_4 ) $ & $\cdot$\\
$ \alpha( c_0 , b_1 b_2 , b_3 )^{-1} $ & $\cdot$ & $\alpha^( b_1 b_2 , b_3 , d_4 ){-1}  $ & $\cdot$\\
$ \alpha (c_0 , b_1 , b_2 b_3 )$ & $\cdot$ & $\alpha ( b_1 , b_2 b_3 , d_4 ) $ & $\cdot$\\
$ \alpha ( c_0 , b_1 , b_2 )^{-1}$ & $\cdot$ & $\alpha( b_1 , b_2 , b_3 d_4 )^{-1}  $ & $ = 1 .$\\
\\
\end{tabular}
\end{center}
\smallskip



It remains to construct examples of $\Gamma_2$ parcels equipped with partial cocycles.  As in previous similar work (cf. \cite{DPY}) it is possible to harness group cohomology to construct the examples.

Consider a group $G$, and a groupoid $G_2$ with objects $\beta$ and $\delta$ equivalent to $G$ by functors
$V:G_2\rightarrow G$ and $W:G\rightarrow G_2$ (for definiteness, $W$ maps the unique object to $\beta$).

Any $d$-cocycle on $G$ pulls back to a $d$-cocycle on the groupoid $G_2$, which will, in turn, restrict to a $d$-cocycle on any subcategory of $G_2$, thus, provided the subcategory admits a $\Gamma_2$-parcel structure, giving coefficients satisfying the cocycle-type conditions for a partial cocycle other than the one arising from the extended Pachner moves at the defect.

Provided $G$ has both finite subgroups (of course one can use the trivial subgroup) and elements of infinite order normalizing them, it is easy to construct a subcategory of $G_2$ which admits a parcel structure:  Choose two (possibly equal) finite subgroups $G_\beta$ and $G_\delta$, and elements $\phi$ and $\psi$ with $\phi\psi$  of infinite order and normalizing $G_\beta$ and $\psi\phi$ (necessarily also of infinite order) normalizing $G_\delta$ and consider the subcategory $\mathcal C$ of $G_2$ generated by $V^{-1}(G_\beta) \cap G_2(\beta, \beta)$, $V^{-1}(G_\delta) \cap G_2(\delta, \delta)$,  $V^{-1}(\phi) \cap G_2(\beta, \delta)$ and $V^{-1}(\psi) \cap G_2(\delta, \beta)$.
 Note that the first two of the generating sets are copies of the finite groups in the local group of the object indicated by the subscript, while the latter two are singletons.

Mapping $G_\beta$ to $Id_\beta$, $G_\delta$ to $Id_\delta$, the chosen preimage of $\phi$ to the generating edge $\iota$ and the chosen preimage of $\psi$ to the generating edge $\omega$ induces a conservative functor from $\mathcal C$ to $\mathcal P(\Gamma_2)$.

To give partial cocycles on such a parcel, it remains only to find $d$-cocycles which satisfy the additional equation.  If $d$ is odd, every $d$-cocyle on $G$ induces a partial cocycle on $\mathcal C$, since the left hand side of the condition corresponding to extended Pachner moves can be multiplied by $\alpha(b_1,\ldots b_d)^{-1}$ and $\alpha(b_1,\dots b_d)^{(-1)^{d+1}} = \alpha(b_1,\dots b_d)$, the first multiplied the factors involving $c_0$ gives 1 by an instance of the $d$-cocycle condition, as does the latter multiplied by the factors involving $d_{d+1}$, thus giving a partial cocycle.

If $d$ is even, this cancellation is impossible, unless the restriction of $\alpha$ to $G_\delta^d$ takes values in $\{+1,-1\}$, in which case the same argument gives a partial cocycle.

Groups with the needed structure are readily available:  choose finite groups $G_\beta$ and $G_\delta$ (ideally with interestingly non-trivial group cohomology) and let $G = G_\beta \times G_\delta \times {\mathbb Z}$, determining its cohomology via the Kunneth formula.  We suspect others can be found in which the infinite degree elements only normalize, but do not centralize, the subgroups, but have not investigated this question.

\section{Prospects for Future Research}

The present work opens up a great many avenues for future research, some arising from questions left uninvestigated in this paper -- describing in detail the fundamental category of directed spaces of the forms $(X,d_{ror}(X))$ and $(X,d_{for}(X))$, ideally in terms related to the stratification from which the directed space structure was constructed, generalizing the constructions of the last section to all $d$-spaces of the form $(X,d_{for}(X))$ -- others arising simply from the possibility of applying the tools of directed algebraic topology to
the study of stratified spaces.

Among the latter, in no particular order, are 

\begin{itemize}
\item Investigation of the higher directed homotopy structure of the directed spaces arising from collared stratifications, both $(X,d_{for}(X))$ and $(X,d_{ror}(X))$.
\item Investigation of $\uparrow \Pi_1(X,d_{source}(X))$ for the stratified spaces in general.  In particular for complex algebraic varieties stratified by iterated singular sets, to what extent does the fundamental category detect the fine structure of the singularities?
\item Investigation of the higher directed homotopy structure of directed spaces given by the source (or target) preorder on a stratified (or even general filtered) space.
\item Application of the directed homology theory of \cite{DGG} to the study of stratified spaces.
\item Generalizing the constructions of TQFTs from higher homotopy types (cf. \cite{P, Y2}) to defect TQFTs using directed space structures associated to the stratification by defects.
\item Investigation of how directed space structures interact with other state-sum constructions of TQFTs (as, for instance Turaev-Viro theory \cite{TV} and Crane-Yetter theory \cite{CY2})
\item Investigation of groups suitable for the construction of $\Gamma_2$-parcels more general than the products considered above.
\end{itemize}

\end{document}